\documentclass{amsart}
\usepackage{amsmath}
\usepackage{amsfonts}
\usepackage{amsthm}
\usepackage{amssymb}
\usepackage{verbatim}
\usepackage{enumerate}
\usepackage{graphicx}
\usepackage{longtable}
\usepackage[all,cmtip]{xy}
\usepackage{upref}
\usepackage{hyperref}

\newtheorem{thm}{Theorem}[section]
\newtheorem{lem}[thm]{Lemma}
\newtheorem{prp}[thm]{Proposition}
\newtheorem{cor}[thm]{Corollary}

\theoremstyle{definition}
\newtheorem{dfn}[thm]{Definition}

\newtheorem{con}[thm]{Conjecture}

\theoremstyle{remark}
\newtheorem{rmk}[thm]{Remark}

\numberwithin{equation}{section}

\newcommand{\R}{\mathbb R}
\newcommand{\Z}{\mathbb Z}

\newcommand{\Pro}{\mathbb P}
\newcommand{\C}{\mathbb C}
\newcommand{\op}[1]{\operatorname{#1}}

\begin{document}

\title{The contact property for nowhere vanishing magnetic fields on $S^2$}

\author{Gabriele Benedetti}
\address{Department of Pure Mathematics and Mathematical Statistics, University of
Cambridge, Cambridge CB3 0WB, UK}
\email{G.Benedetti@dpmms.cam.ac.uk}

\subjclass[2010]{37J05, 37J55, 53D05, 53D10}

\keywords{Dynamical systems, Magnetic field, Symplectic geometry, Contact hypersurfaces, Dynamical convexity}
\date{\today}
\begin{abstract}
In this paper we give some positive and negative results about the contact property for the energy levels $\Sigma_c$ of a symplectic magnetic field on $S^2$. In the first part we focus on the case of the area form on a surface of revolution. We state a sufficient condition for an energy level to be of contact type and give an example where the contact property fails. If the magnetic curvature is positive, the dynamics and the action of invariant measures can be numerically computed. This hints at the conjecture that an energy level of a symplectic magnetic field with positive magnetic curvature should be of contact type.

In the second part we show that, for small energies, there exists a convex hypersurface $N_c$ in $\mathbb C^2$ and a double cover $N_c\rightarrow\Sigma_c$ such that the pull-back of the characteristic distribution on $\Sigma_c$ is the standard characteristic distribution on $N_c$. As a corollary we prove that there are either $2$ or infinitely many periodic orbits on $\Sigma_c$. The second alternative holds if there exists a contractible prime periodic orbit.
\end{abstract}
\maketitle
\begingroup
\let\clearpage\relax
\section{Introduction}\label{sec_int}
Let the triple $(M,g,\sigma)$ represent a magnetic system, where $M$ is a closed manifold, $g$ is a Riemannian metric and $\sigma$ is a closed $2$-form on $M$. A magnetic system gives rise to a Hamiltonian vector field $X^{g,\sigma}$ on the symplectic manifold $(TM,d\alpha-\pi^*\sigma)$, where $\pi$ is the projection from $TM$ to $M$ and $\alpha$ is the pull-back of the Liouville $1$-form $\lambda$ on the cotangent bundle via the duality isomorphism given by $g$. The kinetic energy $E\big((x,v)\big)=\frac{1}{2}g_x(v,v)$ is the Hamiltonian function associated to $X^{g,\sigma}$. The dynamics of the magnetic field received much attention in the last three decades, in particular as regard the existence of periodic orbits. When $M$ is a surface, the two classical approaches that have been pursued are Morse-Novikov theory and symplectic topology (see Ta{\u\i}manov's \cite{tai4} and Ginzburg's \cite{ginsur} surveys for details and further references). More recently many other techniques have been developed. Some of 
them rely on the (weakly) exactness of the magnetic form \cite{bahtai,pol1,mac1,cmp,osu,con,pat2,mer1,frasch2,tai5}. Others seek solutions with low kinetic energy \cite{schl}, the majority of them assuming further that $\sigma$ is symplectic \cite{ker,ginker1,mac2,ginker2,cgk,gg1,ker3,lu,gg2,ush}. Schneider's approach \cite{sch1,sch2,sch3} for orientable surfaces and symplectic $\sigma$ uses a suitable index theory and shows in a very transparent way how the Riemannian geometry of $g$ influences the problem. Finally, we point out \cite{koh} where heat flow techniques are employed and \cite{fmp1,fmp2} which construct a Floer theory for particular magnetic fields.
This paper would like to give its own contribution to the subject by studying some aspects of magnetic dynamics when $M$ is the $2$-sphere and $\sigma$ is symplectic.

\bigskip For a general magnetic system we know that the zero section is the set of rest points for the flow and all the smooth hypersurfaces $\Sigma_c:=\{E=c\}$, with $c>0$, are invariant sets. Hence, we can analyze the dynamics $X^{g,\sigma}\big|_{\Sigma_c}$ separately for every $c$. In particular we are interested to determine whether $\Sigma_c$ is of contact type or not. Hypersurfaces of contact type in symplectic manifolds have been intensively studied in relation to the problem of the existence of closed orbits. After some positive results in particular cases \cite{wei1,rab1,rab2}, in 1978 Alan Weinstein conjectured that every closed hypersurface of contact type (under some additional homological condition now thought to be unnecessary) carries a periodic orbit \cite{wei2}. The conjecture is still open in its full generality, but for magnetic systems on orientable surfaces is a consequence of the solution to the conjecture for closed $3$-manifold by Taubes \cite{tau1}. Moreover, Cristofaro-Gardiner and 
Hutchings have recently improved the lower bound on the number of periodic orbits in dimension $3$ to two \cite{crihut}. The case of irrational ellipsoids in $\C^2$ shows that their estimate is sharp.

The main examples of hypersurfaces of contact type are given by boundaries of star-shaped domains in $\C^n$ and fiberwise star-shaped regions in standard cotangent bundles $(T^*M,d\lambda)$. From this second kind of examples we also deduce that, if $H:T^*M\rightarrow\R$ is a convex superlinear Hamiltonian, its energy levels above the strict Ma\~n\'e critical value are of contact type (see \cite[Corollary 2]{cipp}). In turn exact magnetic systems can be reduced to this class after applying the composition of a translation in the fibers by a primitive of $\sigma$ with the duality isomorphism given by $g$. The strict critical value in this case takes the form
\begin{equation}
c_0(g,\sigma):=\inf_{\{\beta\, |\, d\beta=\sigma\}}\sup_{x\in M}\frac{1}{2}g_x(\beta_x,\beta_x). 
\end{equation}
Hence, if $c>c_0(g,\sigma)$, $\Sigma_c$ is of contact type. If $M$ is an orientable surface different from the $2$-torus, in \cite{cmp} the converse implication is also proved. Namely, $\Sigma_c$ is not of contact type if $c\leq c_0(g,\sigma)$. One of the goals of this paper is to understand how the picture changes on orientable surfaces when $\sigma$ is not exact, starting by analyzing the case of a symplectic $\sigma$ on $S^2$. Henceforth, we will always be under such hypotheses unless stated otherwise.

In this setting we claim that $d\alpha-\pi^*\sigma$ is still exact on $\Sigma_c$. If $K$ is the Gaussian curvature of $g$ and $\mu$ its Riemannian volume, we define $\overline{\sigma}:=\sigma-\frac{[\sigma]}{4\pi}K\mu$, where $[\sigma]=\int_{S^2}\sigma$. Then, $\overline{\sigma}$ is exact and any $1$-form $\beta$ such that $d\beta=\overline{\sigma}$ yields a primitive $\alpha-\pi^*\beta+\frac{[\sigma]}{4\pi}\frac{\psi}{2c}$ of $(d\alpha-\pi^*\sigma)\big|_{\Sigma_c}$, where $\psi$ is the Levi Civita connection form.
Using this kind of primitives in Proposition \ref{boundi} we find a subset of energies $\widehat{\mathcal C}(g,\sigma)$, where the contact property holds. If $f:S^2\rightarrow \R$ is defined by $\sigma=f\mu$ and we set $m(g,\sigma):=\sqrt{2c_0(g,\overline{\sigma})}$, then $\widehat{\mathcal C}(g,\sigma)=(0,+\infty)$ provided $m(g,\sigma)^2<\frac{[\sigma] \inf f}{\pi}$ and $(0,\frac{1}{2}m^-(g,\sigma)^2)\cup(\frac{1}{2}m^+(g,\sigma)^2,+\infty)\subset \widehat{\mathcal C}(g,\sigma)$ otherwise. Here $m^-(g,\sigma)$ and $m^+(g,\sigma)$ are the two roots of the following quadratic equation in the variable $m$:
\begin{equation}
m^2-m(g,\sigma)m+\frac{[\sigma] \inf f}{4\pi}=0.
\end{equation}
A first natural choice, when $K>0$, would be to take $\sigma=K\mu$. In this case $m(g,K\mu)=0$ and every energy level is of contact type. A second natural choice, without any assumption on $K$, would be to take $\sigma=\mu$. In this case the inclusion of the two intervals in $\widehat{\mathcal C}(g,\mu)$ is an equality, so that we have a precise description of this set, once we know $m(g,\mu)$.

In Section \ref{sec_rev} we carry out such computation for surfaces of revolution highlighting the relation between $m(g,\mu)$ and geometric properties of $(S^2,g)$. For example, in Proposition \ref{curvin} we show that $m(g,\mu)^2<\frac{[\mu]}{\pi}$, and hence $\widehat{\mathcal C}(g,\mu)=(0,+\infty)$, provided the surface of revolution is symmetric with respect to the equator and the curvature increases when we move from the poles to the equator. On the other hand, we will see that, for surfaces with normalized area, $m(g,\mu)$ can be arbitrarily big if the curvature is sufficiently concentrated around at least one of the poles. This opposite behaviour yields the following result.
\begin{prp}\label{bigm}
For every $C>0$, there exists a convex surface with total area $4\pi$ such that $m(g,\mu)>C$.
\end{prp}
At this point one would like to understand how good is the set $\widehat{\mathcal C}(g,\mu)$ in approximating the actual set of energies where the contact property holds. For this purpose we employ McDuff's criterion \cite{mcd}, which says that $\Sigma_c$ is of contact type provided all the $X^{g,\sigma}$-invariant measures supported on this hypersurface have positive action. Finding the actions of an invariant measure is usually a difficult task. However, for surfaces of revolution there are always some latitudes that are the supports of periodic orbits. We compute the action of such latitudes in Proposition \ref{paract}. If $\jmath:TS^2\rightarrow TS^2$ is the complex structure induced by $g$, we can define the magnetic curvature $K_c:\Sigma_c\rightarrow\R$ as $K_c:=2cK-\sqrt{2c}(df\circ\jmath)+f$. In Proposition \ref{reddyn} we prove that if $K_c>0$, we only have two periodic orbits which are latitudes and, by Proposition \ref{paract}, their action is positive. Therefore, they do not represent an 
obstruction to the contact property. In addition, under the same curvature assumption, we are able to give a simple description of the dynamics of the symplectic reduction of the system induced by the rotational symmetry. In particular, this allows us to devise a numerical strategy to compute the action of all the ergodic invariant measures as we explain in Section \ref{sub_act}. The data we have collected suggest that all such actions are positive, hinting, therefore, at the following conjecture.
\begin{con}\label{conj}
Let $(S^2,g,\sigma)$ be a symplectic magnetic system and suppose that for some $c>0$, the magnetic curvature is positive. Then, $\Sigma_{c}$ is of contact type.
\end{con}
The numerical computations, and possibly an affirmative answer to the conjecture, would then indicate that the system $(S^2,g,\mu)$ associated to a convex surface of revolution would be of contact type at every energy level and show that $\widehat{\mathcal C}(g,\mu)$ does not describe the actual set of contact type energies in the cases provided by Proposition \ref{bigm}. We also remark that establishing the conjecture will yield another proof of Corollary 1.3 in \cite{sch2} about the existence of two closed orbits on every energy level, when $K\geq0$ and $f>0$.

To complete the picture, in Proposition \ref{nonnec} we construct energy levels of contact type without the positive magnetic curvature assumption and, using again Proposition \ref{paract}, in Proposition \ref{noncon} we give an example of an energy level which is not of contact type.

\bigskip In the last part of the paper we drop the rotational symmetry assumption and we focus on low energy levels. By the previous discussion, we know that they are of contact type. Moreover, they are diffeomorphic to the real projective space so that we have double covers $p_c:S^3\rightarrow\Sigma_c$. Thus, we can pull back the contact form on $\Sigma_c$ to a contact form on $S^3$. In \cite{hwz1}, Hofer, Wysocki and Zehnder singled out the subclass of \textit{tight} and \textit{dynamically convex} contact form on $S^3$ for which the associated Reeb dynamics has a simple qualitative description via a disk-like surface of section. We shall prove in Section \ref{sec_dyn2} that our pull-back forms fall into this class.
\begin{prp}\label{corfinalo}
If the energy $c$ is low enough, there exists a contact form $\tau_c$ on $\Sigma_c$ and a covering map $p_c:S^3\rightarrow \Sigma_c$ such that
\begin{equation}
\bullet\ d\tau_c=(d\alpha-\pi^*\sigma)\big|_{\Sigma_c}\quad\bullet\ p_c^*\tau_c\mbox{ is tight and dynamically convex}.
\end{equation}
\end{prp}
\begin{rmk}
As we will observe later, $\tau_c$ is tight and dynamically convex if and only if $p_c^*\tau_c$ is. Hence, the statement does not depend on the choice of $p_c$. We also notice that this result is analogous to the dynamical convexity of geodesic flows of suitably pinched Finsler metrics on $S^2$, proved in \cite{pathar}. In the magnetic case low kinetic energy plays the same role as the pinching condition in the Finsler case when it comes to estimating the linearization of the flow. Further interplay between contact and Finsler geometry has been explored in \cite{hrysal1}.
\end{rmk}
We give two independent proofs of Proposition \ref{corfinalo}. The first one constructs a $p_c$ such that $p_c^*\tau_c$ is the contact form induced by a convex embedding of $S^3$ in $\C^2$. Then, one uses that the convexity of the hypersurface implies the dynamical convexity of the system \cite[Theorem 3.7]{hwz1}.

The second one does not construct an explicit $p_c$ but proves directly that $\tau_c$ is dynamically convex. This method has two advantages. First, the quantitative estimates we get will be related to the geometry of $(S^2,g,\sigma)$ in a more transparent way. Second, we can adapt the computations on the Conley-Zehnder index, needed to show dynamical convexity, to the case of a surface $M$ with genus bigger than one as follows.

Let $\Sigma_c$ be a sufficiently low energy level of a symplectic magnetic system on such $M$. Then, as happens for the $2$-sphere, $\Sigma_c$ is of contact type and the associated contact structure is homotopically trivial. Therefore, we can define the Conley-Zehnder index $\mu_{\op{CZ}}$ of a null-homologous periodic orbit $\gamma$ on $\Sigma_c$. As we explain in Remark \ref{gengen}, one can show that
\begin{equation}\label{gendyn}
 \mu_{\op{CZ}}(\gamma)\leq 2\chi(M)+1.
\end{equation}
Macarini and Abreu recently presented at the \textit{``Workshop on conservative dynamics and symplectic geometry'' (Rio de Janeiro, 2013)} a generalization of the notion of dynamical convexity, which is relevant to this setting. Thanks to Inequality \eqref{gendyn} they can show that symplectic magnetic systems on orientable surfaces of genus bigger than one are dynamically convex (in this broader sense) on low energy levels. We refer the reader to their forthcoming work for further details and applications of this generalized notion to dynamics.

Coming back to the case of the $2$-sphere, by \cite{hwz1} dynamical convexity yields the existence of a disk-like surface of section for the lifted flow. We remark, moreover, that this result holds even if the contact form is nondegenerate. At this point, one could ask if the projection of the open disk to $\Sigma_c$ is still an embedding. If this was the case, we would have a well-defined Poincar\'e return map for the original system, and we could apply known results for area-preserving maps \cite{bro,fra2,fra3} to obtain directly Theorem \ref{mainthm} below. Actually, Proposition 1.8 in \cite{hls} guarantees that there is a choice of a disk-like surface of section for the lifted flow with this property provided the contact form is \textit{nondegenerate}. To deal with the general case, we have to use the following proposition, which still enables us to transfer some dynamical information from the lifted flow to the original one. Its proof uses Lemma \ref{lemlin} about the linking number of closed orbits of 
the lifted flow.
\begin{prp}\label{covthe}
Let $N$ be diffeomorphic to $\mathbb R\mathbb P^3$ and consider a double cover $p:S^3\rightarrow N$. Let $Z\in\Gamma(N)$ be a vector field preserving some volume form on $N$ and let $\widehat{Z}\in\Gamma(S^3)$ be its lift. If $\widehat{Z}$ has a disk-like surface of section, then $Z$ has either two or infinitely many periodic orbits. The second alternative holds if, in addition, $Z$ has a prime contractible periodic orbit.

Let $\tau\in\Omega^1(N)$ be a contact form such that $p^*\tau$ is tight and dynamically convex. Then, the hypotheses of the proposition are satisfied taking as $Z$ the Reeb vector field of $\tau$.  
\end{prp}

Combining together Proposition \ref{corfinalo} and Proposition \ref{covthe} we obtain the main result of this paper.
\begin{thm}\label{mainthm}
Let $\sigma$ be a symplectic magnetic form on the Riemannian manifold $(S^2,g)$. Then, if the energy $c$ is low enough, the flow of $X^{g,\sigma}\big|_{\Sigma_c}$ has either two or infinitely many periodic orbits. In particular the second alternative holds if there exists a prime contractible periodic solution on $\Sigma_c$.
\end{thm}
\begin{rmk}
Looking at the magnetic system given by the round metric and its volume form, we draw two conclusions. First, that a prime contractible orbit does not need to exist. Second, that there are dynamically convex magnetic systems with infinitely many closed orbits on every energy level. On the other hand, we do not know any instance of an energy level of a symplectic magnetic system with exactly $2$ periodic orbits. However, \cite[Theorem 1.3]{sch1} shows that there are examples with exactly $2$ periodic orbits, whose projection on $S^2$ is a simple curve.  
\end{rmk}
\subsection*{Acknowledgements}
I am grateful to my advisor Gabriel Paternain for many helpful discussions about the dynamics of magnetic field, for his invaluable support in preparing this paper and for having suggested to me, in the first place, that the results in \cite{pathar} could be generalized to this setting. I would like to thank Leonardo Macarini for having pointed out to me a mistake in Inequality \eqref{gendyn} contained in a previous extended version of this work. Finally, I am indebted to Marco Golla for the simple and nice proof about the parity of the linking number contained in Lemma \ref{lemlin} and for his genuine interest in my work.
\section{Preliminaries}\label{sec_pre}
In this section we set the notation and recall the prerequisites needed in the subsequent discussion. The first subsection  describes the conventions and symbols used in the paper. The second subsection is devoted to the basic properties of the tangent bundle of an oriented Riemannian $2$-sphere and of magnetic fields. Finally, the third subsection deals with exact Hamiltonian structures and their relationship with contact geometry.

\subsection{General notation}
Throughout the paper all objects are supposed to be smooth. If $M$ is a manifold we denote by $\Omega^k(M)$ the space of $k$-differential forms on $M$ and by $\Gamma(M)$ the space of vector fields on $M$. The interior product between $Z\in\Gamma(M)$ and $\omega\in\Omega^k(M)$ will be written as $\imath_Z\omega$. If $\omega\in\Omega^k(M)$, we denote by $\mathcal P^{\omega}$ the set of its primitives. Namely, $\mathcal P^{\omega}=\{\tau\in\Omega^{k-1}(M)\ |\ \omega=d\tau\}$.

If $Z\in\Gamma(M)$, we denote by $\mathcal L_Z$ the associated Lie derivative and by $\Phi^Z$ the flow of $Z$ defined on some subset of $\mathbb R\times M$. We write its time $t$ flow map as $\Phi^Z_t$. Moreover, we call $\Omega^k_{Z}(M)$ the space of $Z$-invariant $k$-forms. In other words $\tau$ belongs to $\Omega^k_Z(M)$ if and only if $\mathcal L_Z\tau=0$.

If $(M,\omega)$ is a symplectic manifold and $H:M\rightarrow \R$ is a real function, the Hamiltonian vector field $X_H$ is defined by $\imath_{X_H}\omega=-dH$.

We define now some objects on $\C^n\simeq \R^{2n}$. Denote by $J_{\op{st}}$ the \textit{standard complex structure} and by $g_{\op{st}}$ the \textit{Euclidean inner product}. Define the \textit{standard Liouville form} $\lambda_{\op{st}}\in\Omega^1(\C^n)$ as $(\lambda_{\op{st}})_z(W):=g_{\op{st}}(J_{\op{st}}(z),W)$, for $z\in\C^n$ and $W\in T_z\C^n\simeq \C^n$. Finally, the \textit{standard symplectic form} is defined as $\omega_{\op{st}}:=\frac{1}{2}d\lambda_{\op{st}}$, or in standard real coordinates $\omega_{\op{st}}=\sum_idx^i\wedge dy^i$.

If $\gamma$ and $\gamma'$ are two knots in $S^3$ we denote by $\op{lk}(\gamma,\gamma')$ their linking number.

\subsection{The geometry of an oriented Riemannian 2-sphere}\label{sub_geo}
Let $\pi:TS^2\rightarrow S^2$ be the tangent bundle of $S^2$ and let $g$ be a Riemannian metric on it. This yields a duality isomorphism $\flat:TS^2\rightarrow T^*S^2$ which we use to push forward the metric for tangent vectors to a dual metric $g$ for $1$-forms. We write $\vert\cdot\vert$ for the induced norms on each $T_xS^2$ and $T^*_xS^2$. From the duality construction we have the identity $\vert l\vert=\sup_{\vert v\vert=1}|l(v)|$ for every $l\in T_x^*S^2$. We collect this family of norms together to get a supremum norm $\Vert\cdot\Vert$ for sections:
\begin{equation}\label{norms}
\forall Z\in \Gamma(S^2),\ \ \Vert Z\Vert:=\sup_{x\in S^2}\vert Z_x\vert;\quad\quad \forall \beta\in \Omega^1(S^2),\ \ \Vert \beta\Vert:=\sup_{x\in S^2}\vert \beta_x\vert.
\end{equation}
The Riemannian metric induces a \textit{kinetic energy function} $E:TS^2\rightarrow \R$ defined by $E\big((x,v)\big):=\frac{1}{2}g_x(v,v)$. The level sets $\Sigma_c:=\{E=c\}\subset TS^2$ are such that
\begin{itemize}
 \item the zero level $\Sigma_0$ is the zero section $\{(x,0)\,|\,x\in S^2\}$;
 \item for $c>0$, $\Sigma_c\rightarrow S^2$ is an $S^1$-bundle with total space diffeomorphic to $\R\mathbb P^3$.
\end{itemize}

Consider $\nabla$ the Levi Civita connection of $g$. Let $\frac{\nabla^\gamma}{dt}$ be the associated covariant derivative along a curve $\gamma$ on $S^2$ and denote by $K$ the Gaussian curvature of $g$. For every $(x,v)\in TS^2$, $\nabla$ induces a horizontal lift $L^{\mathcal H}_{(x,v)}:T_xS^2\rightarrow T_{(x,v)}TS^2$. It has the property that $d_{(x,v)}\pi\circ L^{\mathcal H}_{(x,v)}=\op{Id}_{T_xS^2}$. The geodesic equation $\frac{\nabla^\gamma}{dt}\dot{\gamma}=0$ for curves $\gamma$ in $S^2$ gives rise to a flow on $TS^2$. The associated vector field $X\in \Gamma(TS^2)$ is called the \textit{geodesic vector field} and can be equivalently defined as $X_{(x,v)}=L^{\mathcal H}_{(x,v)}(v)$.

Suppose that we also have fixed an orientation $\mathfrak o$ on $S^2$. Then, we combine it with the Riemannian metric $g$ to define the positive Riemannian volume form $\mu\in\Omega^2(S^2)$ and the $2\pi$-periodic flow $\Phi^V_\varphi:TS^2\rightarrow TS^2$, which rotates every fiber by an angle $\varphi$. If we denote by $\jmath$ the rotation of $\pi/2$, then the generator $V$ of this flow at $(x,v)$ is the verical lift of $\jmath_x(v)$. Any orbit of $\Phi^V$ is closed and its support is $\Sigma_c\cap T_xS^2$, for some $x\in S^2$. In particular, $\Phi^V$ leaves every energy set $\Sigma_c$ invariant and, hence, $V\big|_{\Sigma_c}\in \Gamma(\Sigma_c)$.

Finally, we use $\jmath$ and the horizontal lift to add a last distinguished vector field $H\in\Gamma(TS^2)$. It is defined as $H_{(x,v)}:=L_{(x,v)}^{\mathcal H}(\jmath_x (v))$.

Take now a $\sigma\in\Omega^2(S^2)$ and construct the symplectic $2$-form $\omega_\sigma:=d\alpha-\pi^*\sigma\in\Omega^2(TS^2)$, where $\alpha$ is the pull-back of the \textit{standard Liouville form} on $T^*S^2$ via $\flat$. We call $\sigma$ a \textit{magnetic form} and the triple $(S^2,g,\sigma)$ a \textit{magnetic system}. There exists a unique real function $f:S^2\rightarrow \R$ called the \textit{magnetic strength} such that $\sigma=f\mu$. In the following discussion, we also suppose that the magnetic system is \textit{symplectic}, namely that $\sigma$ is a symplectic form. By inverting the orientation, we can assume that the magnetic form is positive with respect to $\mathfrak o$. 

We define the \textit{magnetic vector field} $X^{g,\sigma}\in\Gamma(TS^2)$ as the $\omega_\sigma$-Hamiltonian vector field associated to $E$ and we refer to $\Phi^{X^{g,\sigma}}$ as the \textit{magnetic flow}. As the geodesic vector field, $X^{g,\sigma}$ comes from a second order ODE for curves in $S^2$: 
\begin{equation}\label{lorsig}
\frac{\nabla^\gamma}{dt}\dot\gamma=F_\gamma\big(\dot{\gamma}\big).
\end{equation}
Here $F:TM\rightarrow TM$ is the \textit{Lorentz force} of the magnetic system. It is a bundle map given by $g(F_x(v),w):=\sigma_x(v,w)$ and can be expressed using the magnetic strength as $F_x(v)=f(x)\jmath_x(v)$. Using the relation between the Levi Civita connection and the horizontal lifts, one finds that $X^{g,\sigma}=X+fV$.

Since $E$ is a integral of motion for the magnetic flow, $\Sigma_0$ is the set of rest points for $X^{g,\sigma}$ and, for $c>0$, $X^{g,\sigma}$ restricts to a nowhere vanishing vector field on $\Sigma_c$.

Scalar multiplication along the fibers $(x,v)\mapsto(x,\frac{v}{\sqrt{2c}})$ sends $\Sigma_c$ to $SS^2:=\Sigma_{1/2}$. The push-forward of $X^{g,\sigma}\big|_{\Sigma_c}$ with respect to this diffeomorphism is $\sqrt{2c}X+fV$. Thus, the dynamics of the magnetic flow is encoded by the $1$-parameter family of vector fields $X^m:=\big(mX+fV\big)\big|_{SS^2}\in\Gamma(SS^2)$, where the relation between the parameters $m$ and $c$ is given by $m(c)=\sqrt{2c}$. For every $m$ we define the \textit{magnetic curvature} function $K_m:SS^2\rightarrow\R$ as $K_m:=m^2K-m(df\circ\jmath)+f$.

From now on we also assume that $\sigma$ (or equivalently $f$) is rescaled by a positive constant in such a way that
\begin{equation}\label{normi}
 \int_{S^2}\sigma=4\pi.
\end{equation}
This operation will only induce a corresponding rescaling of the parameter $m$ and, hence, will not affect our study.

From the discussion above, we see that we can restrict our attention to the geometry of $SS^2$. Here we have a canonical frame, given by $(X,V,H)\big|_{SS^2}$, which induces the dual coframe $(\alpha,\psi,\eta)$. We have the following three bracket relations and dual differential relations:
\begin{equation}\label{bra}
\left\{\begin{array}{ccc}
d\alpha=\psi\wedge\eta,& d\psi=K\eta\wedge\alpha=-K\pi^*\mu,& d\eta=\alpha\wedge\psi.\\
\mbox{}[V,X]=H,&[H,V]=X,&[X,H]=KV. 
\end{array}\right.
\end{equation}

The frame also yields a volume form inducing an orientation $\mathfrak O_{SS^2}$. It is called the \textit{Liouville volume form} $\nu\in\Omega^3(SS^2)$ and it is defined as $\nu:=\alpha\wedge\psi\wedge\eta$. The orientation $\mathfrak O_{SS^2}$ is obtained from the one on $TS^2$ induced by $\omega_\sigma$ following the convention of putting the outward normal to $SS^2$ first.

It is easy to define a \textit{$C^0$-topology} and a \textit{$C^1$-topology} for elements $Z$ of $\Gamma(SS^2)$. The former is given by the uniform convergence of the three functions $\alpha(Z)$, $\psi(Z)$ and $\eta(Z)$. The latter also requires the uniform convergence of the derivatives of these functions along $X$, $V$ and $H$. With this definition, the Lie bracket is a continuous map from $\Gamma(SS^2)\times\Gamma(SS^2)$ endowed with the product $C^1$-topology and $\Gamma(SS^2)$ endowed with the $C^0$-topology.

Pulling back $\omega_\sigma$ on $SS^2$ using multiplication along the fibers again, we also get the family of nowhere vanishing closed $2$-forms $\omega_m:=md\alpha-\pi^*\sigma\in\Omega^2(SS^2)$. These are instances of what is generally called an \textit{Exact Hamiltonian Structure} (or EHS for brevity), which we describe in the next subsection.

\subsection{Exact Hamiltonian structures}
Let $(N,\mathfrak O)$ be a three-manifold $N$ endowed with an orientation $\mathfrak O$.
\begin{dfn}
A closed and nowhere vanishing $2$-form on $(N,\mathfrak O)$ is called a \textit{Hamiltonian Structure}. If the form is further assumed to be exact, it is called an \textit{Exact Hamiltonian Structure}. 
\end{dfn}
Every EHS $\omega$ yields the one-dimensional oriented foliation $\ker \omega$. In general one is interested to study its \textit{qualitative dynamics}, namely the flow of any positive section of $\ker \omega$, up to time reparametrization. Many results on the dynamics depend on finding a special primitive for $\omega$.
\begin{dfn}
We say that $\omega$ is \textit{of contact type} if there exists a contact form $\tau\in\mathcal P^\omega$. If we denote by $R^\tau$ the \textit{Reeb vector field} of $\tau$, then one among $R^\tau$ and $-R^\tau$ is a positive section of $\ker\omega$. We say that $\omega$ is of \textit{positive} or \textit{negative} contact type accordingly. 
\end{dfn}
\begin{rmk}\label{conpro}
If we fix any positive section $Z$ of $\ker\omega$, being of positive (respectively negative) contact type is equivalent to finding $\tau\in\mathcal P^\omega$ such that $\tau(Z):N\rightarrow\R$ is a positive (respectively negative) function. In this case $R^{\tau}=\frac{Z}{\tau(Z)}$.
\end{rmk}
The most direct way to detect the contact property is to use Remark \ref{conpro}. However, this method could be difficult to apply, especially if we want to prove that an EHS is not of contact type, since we should check that every function $\tau(Z)$ vanishes at some point. This problem is overcome by the following necessary and sufficient criterion contained in McDuff \cite{mcd}.
\begin{prp}\label{mcdcri}
Let $\omega$ be an EHS and $Z$ be a positive section of $\ker\omega$. Then, $\omega$ is of positive (respectively negative) contact type if and only if the action of every null-homologous $Z$-invariant measure is positive (respectively negative).
\end{prp}
We recall that a $Z$-invariant measure $\zeta$, is a Borel probability measure on $N$, such that
\begin{equation}
\forall\, h:N\rightarrow\R,\quad \int_Ndh(Z)\zeta=0.
\end{equation}
We associate to $\zeta$ an element $\rho(\zeta)$ in $H^1(N,\R)^*=H_1(N,\R)$ defined as
\begin{equation}
\forall\, [\beta]\in H^1(N,\R),\quad <\rho(\zeta),[\beta]>:=\int_N\beta(Z)\zeta.
\end{equation}
Suppose $Z$ is a positive section of $\ker\omega$, with $\omega$ an EHS, $\zeta$ is null-homologous (namely $\rho(\zeta)=0$) and $\tau\in\mathcal P^\omega$. We define the \textit{action} of $\zeta$, which does not depend on $\tau$, as
\begin{equation}
\mathcal A(\zeta):=\int_N\tau(Z)\zeta.
\end{equation}

The importance of being of contact type relies on the fact that we can use results for Reeb flows to understand the qualitative dynamics of an EHS. Besides the solution of the Weinstein conjecture in dimension $3$, which we discussed briefly in the introduction, we are interested in the work of Hofer, Wysocki and Zehnder \cite{hwz1}. To state it we need to define the Conley-Zehnder index of a closed Reeb orbit. We refer to \cite{hofkri} for the proofs and further details.

\subsection{Reeb vector fields and the Conley-Zehnder index}\label{sub_czi}
We start with the definition of the index for a path with values in $\op{Sp}(1)$, the group of $2\times 2$-symplectic matrices. For any $T>0$, we set $\op{Sp}_T(1):=\{\Psi:[0,T]\rightarrow \operatorname{Sp}(1)\ |\ \Psi(0)=\operatorname{Id}\}$. We call $\Psi\in \op{Sp}_T(1)$ \textit{non-degenerate} if $\Psi(T)$ does not have $1$ as eigenvalue.

Given $\Psi\in\op{Sp}_T(1)$, we associate to every $u\in\mathbb R^2\setminus \{0\}$ a \textit{winding number} $\Delta\theta(\Psi,u)$ as follows. Identify $\R^2$ with $\C$ and let
\begin{equation}\label{teta}
\frac{\Psi(t)u}{|\Psi(t)u|}=e^{i\theta^\Psi_u(t)},
\end{equation}
for some function $\theta^\Psi_u:[0,T]\rightarrow\R$. We define $\Delta\theta(\Psi,u):=\theta^\Psi_u(T)-\theta^\Psi_u(0)$. Let
\begin{equation}
I(\Psi):=\left\{\frac{1}{2\pi}\Delta\theta(\Psi,u)\ \Big|\ u\in\R^2\setminus \{0\}\right\}.
\end{equation}
The interval $I(\Psi)$ is closed and its length is strictly less than $1/2$. We notice that the set $e^{2\pi i I(\Psi)}\subset S^1$ is completely determined by the endpoint $\Psi(T)$. In particular, we see that $\Psi$ is nondegenerate if and only if $\mathbb Z\cap \partial I(\Psi)=\emptyset$. We define the Conley-Zehnder index for the non-degenerate case as
\begin{equation}
\mu_{\op{CZ}}(\Psi):=\left\{\begin{array}{cc}
 2k,&\mbox{if }k\in I(\Psi),\mbox{ for some }k\in\Z;\\
2k+1,&\mbox{if }I(\Psi)\subset(k,k+1),\mbox{ for some }k\in\Z.
\end{array}\right.
\end{equation}
Then, we extend the definition to the degenerate case by taking the maximal lower semicontinuous extension. This amounts to using the same recipe as in the non-degenerate case, but for an interval $I(\Psi)-\varepsilon$, shifted to the left by an arbitrary small amount. With this definition we have that, for any $k\in\Z$,  
\begin{equation}\label{rmkdyn}
I(\Psi)\subset(k,+\infty)\quad\iff \quad \mu_{\op{CZ}}(\Psi)\geq 2k+1.
\end{equation}

We move now to describe the Conley-Zehnder index of a closed Reeb orbit. Suppose we have a three-manifold $N$ with a contact form $\tau$ with induced contact structure $\xi:=\ker\tau$ and Reeb vector field $R^\tau$. Assume that the first Chern class $c_1(\xi)\in H^2(N,\mathbb Z)$ vanishes on $\pi_2(N)$. Let $\gamma:\R/T\Z\rightarrow N$ be a contractible periodic orbit of $R^\tau$, which bounds a disk $i:D\rightarrow N$. Let $(\epsilon^2_D,\omega_{\op{st}})$ be the trivial symplectic bundle of rank $2$ on $D$ and let $\chi:(i^*\xi,i^*d\tau)\rightarrow(\epsilon^2_D,\omega_{\op{st}})$ be a $d\tau$-symplectic trivialization of $i^*\xi$ on $D$. Since $\xi$ is invariant under the flow of $R^\tau$, we can form the path of symplectic matrices $\Psi^{D,\chi}_\gamma\in\op{Sp}_T(1)$ given by
\begin{equation}
\Psi^{D,\chi}_\gamma(t):=\chi_{\gamma(t)}\circ d_{\gamma(0)}\Phi^{R^\tau}_t\circ \chi^{-1}_{\gamma(0)}\in \op{Sp}(1).
\end{equation}
\begin{dfn}
The \textit{Conley-Zehnder index} of $\gamma$ is defined as $\mu_{\op{CZ}}(\gamma):=\mu_{\op{CZ}}(\Psi^{D,\chi}_\gamma)$. The hypothesis on the Chern class ensures that this number does not depend on the pair $(D,\chi)$.

A contact form $\tau$ on $N$ such that $c_1(\ker\tau)\big|_{\pi_2(N)}=0$ is said to be \textit{dynamically convex} if, for every contractible periodic Reeb orbit $\gamma$, $\mu_{\op{CZ}}(\gamma)\geq3$.
\end{dfn}
The main result in \cite{hwz1} can be stated as follows.
\begin{thm}[Hofer, Wysocki and Zehnder, 1998]\label{hwzthm}
The Reeb flow of a tight dynamically convex contact form on $S^3$ admits a disk-like surface of section. 
\end{thm}
We recall that if $Z$ is a vector field on a closed $3$-manifold $N$, a \textit{global surface of section} for $Z$ is an embedded compact surface $i:S\hookrightarrow N$ such that
\begin{itemize}
 \item the boundary of $S$ is the union of supports of periodic orbits for $Z$;
 \item the vector field $Z$ is transverse to $\dot{S}:=S\setminus \partial S$;
 \item every flow line hits the surface in forward and backward time.
\end{itemize}
Under these hypotheses we can define a first return map $F_{\dot{S}}:\dot{S}\rightarrow\dot{S}$, whose discrete dynamics carries important information about the qualitative dynamics of $Z$. When $S$ is a disk, \cite[Proposition 5.4]{hwz1} implies that $F_{\dot{S}}$ is $C^0$-conjugated to a homeomorphism of the disk preserving the standard Lebesgue measure. Therefore, one can use the work of Brouwer (\cite{bro}) and Franks (\cite{fra2} and \cite{fra3}) on area-preserving homeomorphisms of planar domains and get the following corollary.
\begin{cor}\label{hwzcor}
Suppose $Z\in\Gamma(S^3)$ is a volume-preserving vector field having a disk-like surface of section. Then, $\Phi^Z$ has $2$ periodic orbits $\gamma_1$ and $\gamma_2$ which form a Hopf link (namely they are unknotted and the absolute value of their linking number is $1$). Either these are the only periodic orbits or there are infinitely many of them. In particular the second case holds if there exists a knotted periodic orbit or if there are two periodic orbits $\gamma$ and $\gamma'$ such that $|\op{lk}(\gamma,\gamma')|\neq 1$.

The hypotheses of the corollary are satisfied if we take $Z=R^\tau$, where $\tau$ is a tight dynamically convex contact form on $S^3$.
\end{cor}
\section{Energy levels of contact type}\label{sec_con}
With the next proposition we resume our discussion about the magnetic flow showing that $\omega_m$ is, indeed, an EHS and that the dynamics of the underlying foliation is given by the magnetic flow. 
\begin{prp}\label{magehs}
The $2$-form $\omega_m$ is an EHS on $(SS^2,\mathfrak O_{SS^2})$ and the magnetic vector field $X^m$ is a positive section of $\ker\omega_m$.
\end{prp}
\begin{proof}
We know that $\omega_m$ is closed and nowhere vanishing. It is also exact since $H^2(SS^2,\R)=0$. The exactness can also be proven in the following way, which has the advantage of exhibiting a distinguished class of primitives.

Let $\overline{\sigma}:=\sigma-K\mu\in\Omega^2(S^2)$. It is an exact form by the Gauss-Bonnet theorem: 
\begin{equation}
\int_{S^2}\overline{\sigma}=\int_{S^2}\sigma-\int_{S^2}K\mu=4\pi-4\pi=0.
\end{equation}
Hence, $\pi^*\sigma=\pi^*\overline{\sigma}+\pi^*K\mu=\pi^*\overline{\sigma}-d\psi$, so that $\pi^*\sigma$, and hence $\omega_m$, is exact. In particular, we have the injections $\mathcal P^{\overline{\sigma}}\hookrightarrow\mathcal P^{\omega_m}$: $\beta\mapsto\tau^{m,\beta}:=m\alpha-\pi^*\beta+\psi$.

To prove that $X^m$ is a positive section of $\ker\omega_m$, it is enough to prove that $X^\sigma\big|_{\Sigma_c}$ is a positive section of $\ker\omega_\sigma\big|_{\Sigma_c}$. This last statement is true because of our choice of the orientation on $SS^2$ and the fact that $X^\sigma$ is the $\omega_\sigma$-Hamiltonian vector field of $E$. 
\end{proof}

Thanks to Proposition \ref{magehs}, it is meaningful to define the set $\op{Con}(g,\sigma)\subset(0,+\infty)$ of all the values $m$ such that $\omega_m$ is of contact type.

The first important piece of information about the contact property is that $\omega_m$ cannot be of \textit{negative} contact type.
\begin{prp}\label{nonex}
The Liouville measure $\nu$ is an $X^m$-invariant null-homologous measure and its action is positive. Therefore $\omega_m$ cannot be of negative contact type.
\end{prp}
\begin{proof}
Noticing that $\alpha\wedge\pi^*\sigma=0$, we find that $\nu=\frac{\alpha}{m}\wedge\omega_m$ and, hence, $\imath_{X^m}\nu=\omega_m$. Since $\omega_m$ is exact, this identity gives at once that $\nu$ is $X^m$-invariant and null-homologous (the latter fact could also be deduced from $H_1(SS^2,\mathbb R)=0$). If we fix $\beta\in\mathcal P^{\overline{\sigma}}$, the action is given by
\begin{equation}
\int_{SS^2}\tau^{m,\beta}(X^m)\nu=\int_{SS^2}\Big(m^2-m\beta_x(v)+f(x)\Big)\nu.
\end{equation}
The integral of $\beta_x(v)$ vanishes since $\nu$ is preserved under $(x,v)\mapsto (x,-v)$ but $\beta_x(v)$ changes sign. Since $m^2+f$, which is the remaining part of the integrand, is positive, the action is also positive. Using Proposition \ref{mcdcri} we conclude that $\omega_m$ is not of negative contact type.
\end{proof}
Before we state a proposition giving a sufficient condition for a positive number $m$ to be in $\op{Con}(g,\sigma)$, we give the following definitions:
\begin{eqnarray}
\mathcal C(g,\sigma)&:=&\left\{m>0\ \Big|\ \ m>\inf_{\beta\in\mathcal P^{\overline{\sigma}}}\sup_{x\in S^2}\left(|\beta_x|-\frac{f(x)}{m}\right)\right\},\\
m(g,\sigma)&:=&\inf_{\beta\in\mathcal P^{\overline{\sigma}}}\Vert \beta\Vert,\\
\label{mpm}m^\pm(g,\sigma)&:=&\!\frac{m(g,\sigma)\pm\sqrt{m(g,\sigma)^2-4\inf f}}{2},\ \mbox{if }m(g,\sigma)^2\geq4\inf f.
\end{eqnarray}
The set of energies $\widehat{\mathcal C}(g,\sigma)$ mentioned in the introduction is obtained from $\mathcal C(g,\sigma)$ by the change of parameter $m\mapsto \frac{1}{2}m^2$.

Finally, for every $\beta\in\mathcal P^{\overline{\sigma}}$, we set $h_{m,\beta}:=\tau^{m,\beta}(X^m)$, which is a function on $SS^2$. 
\begin{prp}\label{boundi}
Let $(S^2,g,\sigma)$ be a symplectic magnetic system normalized as in \eqref{normi}. If $m\in\mathcal C(g,\sigma)$, there exists $\beta\in\mathcal P^{\overline{\sigma}}$ such that $\tau^{m,\beta}\in\mathcal P^{\omega_m}$ is a contact form and, hence, $\omega_m$ is of positive contact type. Therefore, 
\begin{equation}
\mathcal C(g,\sigma)\subset \op{Con}(g,\sigma). 
\end{equation}
Furthermore the following inclusions hold 
\begin{eqnarray}
\mbox{if }m(g,\sigma)<2\sqrt{\inf f},& (0,+\infty)\subset\mathcal C(g,\sigma);\\
\mbox{if }m(g,\sigma)\geq2\sqrt{\inf f},& (0,m^-(g,\sigma))\cup(m^+(g,\sigma),+\infty)\subset \mathcal C(g,\sigma).
\end{eqnarray}
\end{prp}
\begin{proof}
Proving that $\tau^{m,\beta}\in\mathcal P^{\omega_m}$ is a contact form is equivalent to showing that $h_{m,\beta}>0$. Since $h_{m,\beta}\big((x,v)\big)=m^2-m\left(\beta_x(v)-\frac{f(x)}{m}\right)$, we have the equivalence
\begin{equation*}
\forall\, (x,v)\in SS^2,\ h_{m,\beta}\big((x,v)\big)>0\ \iff\ \forall x\in S^2,\ m^2-m\left(\vert\beta_x\vert-\frac{f(x)}{m}\right)>0.
\end{equation*}
Taking the infimum over $\beta\in\mathcal P^{\overline{\sigma}}$, we see that if $m\in\mathcal C(g,\sigma)$, then $\omega_m$ is of positive contact type. Furthermore we have
\begin{equation}\label{stimpos}
m^2-m\vert\beta_x\vert+f(x)\geq m^2-m\Vert \beta\Vert+\inf f.
\end{equation}
A positive number $m$ satisfies $m^2-m\Vert \beta\Vert+\inf f>0$ for some $\beta\in\mathcal P^{\overline{\sigma}}$, if and only if $m^2-m(g,\sigma)m+\inf f>0$.
The function $m\mapsto m^2-m(g,\sigma)m+\inf f$ is always positive when $m(g,\sigma)^2<4\inf f$. On the other hand, when $m(g,\sigma)^2\geq4\inf f$, it is positive for $m>m^+(g,\sigma)$ or $m<m^-(g,\sigma)$, where $m^-(g,\sigma)$ and $m^-(g,\sigma)$ (defined as before) are the roots of the second degree equation
\begin{equation}\label{stimeq}
m^2-m(g,\sigma)m+\inf f=0.
\end{equation}
\end{proof}

\begin{rmk}\label{ext}
We can extend the families $X^m$ and $\omega_m$ in $m=0$, getting $X^0=fV$ and $\omega_0=-\pi^*\sigma$. Since $\sigma$ is symplectic, Proposition \ref{magehs} holds also in this case and we find that, for any $\beta\in\mathcal P^{\overline{\sigma}}$, $\tau^{0,\beta}$ is a contact form whose Reeb vector field is $V$. These are exactly the Boothby-Wang contact forms (see \cite{bw}) corresponding to the integral symplectic manifold $(S^2,-\sigma)$.
\end{rmk}
\section{A class of examples: surfaces of revolution}\label{sec_rev}
To construct a surface of revolution, take a function $\gamma:[0,\ell_\gamma]\rightarrow \R$ and consider the conditions:
\begin{enumerate}
 \item $\gamma(t)=0$ if $t=0$ or $t=\ell_\gamma$ and is positive otherwise,
 \item $\dot{\gamma}(0)=1$, $\dot{\gamma}(\ell_\gamma)=-1$ and $|\dot{\gamma}(t)|<1$ for $t\in(0,\ell_\gamma)$,
 \item all even derivatives of $\gamma$ vanish for $t\in\{0,\ell_\gamma\}$,
 \item the following equality is satisfied
\begin{equation}\label{norminto}
\int_0^{\ell_\gamma}\gamma(t)dt=2.
\end{equation}
\end{enumerate}
A function $\gamma$ satisfying the first three hypotheses of the list is called a \textit{profile function}. If also the fourth one holds, we say that $\gamma$ is \textit{normalized}.

Let $S^2_\gamma$ be the quotient of $[0,\ell_\gamma]\times \R/2\pi$ with respect to the equivalence relation that collapses the set $\{0\}\times\R/2\pi$ to a point and the set $\{\ell_\gamma\}\times\R/2\pi$ to another point. We call these points the \textit{south} and \textit{north pole}, respectively. Outside the poles the smooth structure is given by the coordinates $(t,\theta)\in(0,\ell_\gamma)\times \R/2\pi$, which also determine a well-defined orientation on $S^2_\gamma$.

We put on $S^2_\gamma$ the Riemannian metric $g_\gamma$, defined in the $(t,\theta)$ coordinates by the formula $g_\gamma=dt^2+\gamma(t)^2d\theta^2$. This metric extends smoothly to the poles because of conditions $2$ and $3$ listed before. Moreover condition $4$ yields the normalization $\op{vol}_{g_\gamma}(S^2_\gamma)=4\pi$.
Let us denote by $(t,\theta,v_t,v_\theta)$ the coordinates on the tangent bundle. On $SS^2_\gamma$ we define the angular function $\varphi\in\R/2\pi\Z$ by the relations 
\begin{equation}
v_t=\cos\varphi,\quad v_\theta=\frac{\sin\varphi}{\gamma(t)}.
\end{equation}
Then, $(t,\varphi,\theta)$ are coordinates on $SS^2_{\gamma}$, which are compatible with the orientation $\mathfrak O_{SS^2}$ defined in Section \ref{sub_geo}. By writing the Levi Civita connection in coordinates, we can express the frame $(X,V,H)$ in terms of the frame $(\widehat{\partial}_t,\partial_\varphi,\widehat{\partial}_\theta)$ associated to these coordinates and vice versa:
\begin{equation}\label{relaz1}
\left\{\begin{array}{rcl}
X&=&\displaystyle\cos\varphi \widehat{\partial}_t-\frac{\dot{\gamma}\sin\varphi}{\gamma}\partial_\varphi+\frac{\sin\varphi}{\gamma}\widehat{\partial}_\theta\\
V&=&\displaystyle\partial_\varphi\\
H&=&\displaystyle-\sin\varphi \widehat{\partial}_t-\frac{\dot{\gamma}\cos\varphi}{\gamma}\partial_\varphi+\frac{\cos\varphi}{\gamma}\widehat{\partial}_\theta,
\end{array}\right. 
\end{equation}
\begin{equation}\label{relaz2}
\left\{\begin{array}{rcl}
\widehat{\partial}_t&=&\displaystyle\cos\varphi X-\sin\varphi H\\
\partial_\varphi&=&V\\
\widehat{\partial}_\theta&=&\displaystyle\gamma\sin\varphi X+\gamma\cos\varphi H+\dot{\gamma}V.
\end{array}\right.
\end{equation}
We have put a hat on $\widehat{\partial}_t$ and $\widehat{\partial}_\theta$ to distinguish them from the coordinate vectors $\partial_t$ and $\partial_\theta$ associated to the coordinates $(t,\theta)$ on $S^2_\gamma$.

We consider as a magnetic form on the surface $SS^2_\gamma$, the Riemannian volume form $\mu_\gamma$. This is a symplectic form which satisfies normalization \eqref{normi}. In coordinates $(t,\theta)$ we have $\mu_\gamma=\gamma dt\wedge d\theta$. With this choice $f\equiv1$ and $\overline{\mu}_\gamma=(1-K)\mu_\gamma$. We recall that the Gaussian curvature $K$ is given by the formula $K=-\frac{\ddot{\gamma}}{\gamma}$.

We use the notation $X^m_\gamma$ and $\omega^\gamma_m$ to refer to $X^m$ and $\omega_m$ on $SS^2_\gamma$. Moreover we set $m_\gamma:=m(g_\gamma,\mu_\gamma)$, $\mathcal C_\gamma:=\mathcal C(g_\gamma,\mu_\gamma)$ and $\op{Con}_\gamma:=\op{Con}(g_\gamma,\mu_\gamma)$. The quantities in \eqref{mpm} simplify to
\begin{equation}
m^\pm_\gamma:=m^\pm(g_\gamma,\mu_\gamma)=\frac{m_\gamma\pm\sqrt{m_\gamma^2-4}}{2}.
\end{equation}
With these definitions and thanks to the fact that $f$ is constant, Proposition \ref{boundi} reduces to the following relations:
\begin{eqnarray}
\label{conint1}\mathcal C_\gamma&=&(0,+\infty),\quad \mbox{if }m_\gamma<2,\\
\label{conint2}\mathcal C_\gamma&=&(0,m^-_\gamma)\cup(m^+_\gamma,+\infty),\quad \mbox{if } m_\gamma\geq2.
\end{eqnarray}
In particular, the smaller $m_\gamma$ is, the bigger the set $\mathcal C_\gamma$ will be. In the next subsection we compute $m_\gamma$, showing that there exists $\beta^\gamma\in\mathcal P^{\overline{\mu}_\gamma}$ such that $m_\gamma=\Vert\beta^\gamma\Vert$.

\subsection{Estimating the set of energy levels of contact type}\label{sub_inv}
Consider an arbitrary closed Riemannian manifold $(M,g)$. Let $Z\in\Gamma(M)$ be a vector field that generates a $2\pi$-periodic flow of isometries on $M$. The projection operator on the space of $Z$-invariant $k$-forms $\mathfrak M^k_{Z}:\Omega^k(M)\rightarrow \Omega^k_{Z}(M)$ is defined as
\begin{equation}\label{invdef}
\forall \tau\in\Omega^k(M),\quad\mathfrak M^k_Z(\tau):=\frac{1}{2\pi}\int_{0}^{2\pi}(\Phi^Z_{\theta'})^*\tau d{\theta'}.
\end{equation}
\begin{prp}\label{inva}
The operators $\mathfrak M^k_Z$ commute with exterior differentiation:
\begin{equation}
d\circ \mathfrak M^k_Z=\mathfrak M^{k+1}_Z\circ d.
\end{equation}
The projection $\mathfrak M^1_Z$ does not increase the norm $\Vert\cdot\Vert$ defined in \eqref{norms}:
\begin{equation}
\forall \beta\in\Omega^1(M),\quad \Vert \mathfrak M^1_Z(\beta)\Vert\leq\Vert\beta\Vert. 
\end{equation}
\end{prp}
\begin{proof}
For the proof of the first part we refer to \cite[Section: The De Rham groups of Lie groups]{boo}. For the latter statement we use that $\Phi^Z$ is a flow of isometries. We take some $(x,v)\in TM$ and compute
\begin{eqnarray*}
\Big|\mathfrak M^1_Z(\beta)_x(v)\Big|&=&\left|\frac{1}{2\pi}\int_{0}^{2\pi}\beta_{\Phi^Z_{\theta'}(x)}\big(d_x\Phi^Z_{\theta'}v\big) d{\theta'}\right|\\
&\leq&\frac{1}{2\pi}\int_{0}^{2\pi}\Vert \beta\Vert\vert d_x\Phi^Z_{\theta'}v\vert d{\theta'}\\
&=&\frac{1}{2\pi}\int_{0}^{2\pi}\Vert \beta\Vert\vert v\vert d{\theta'}\\
&=&\Vert \beta\Vert\vert v\vert.
\end{eqnarray*}
Hence, $\vert \mathfrak M^1_Z(\beta)_x\vert\leq\Vert \beta\Vert$. Taking the supremum over $x$ in $M$ finishes the proof.
\end{proof}
Let us apply this general result to $S^2_\gamma$. Consider the coordinate vector field $\partial_\theta$. This extends to a smooth vector field also at the poles and $\Phi^{\partial_\theta}$ is a $2\pi$-periodic flow of isometries on the surface. Applying Proposition \ref{inva} to this case, we get the following corollary.
\begin{cor}
The $2$-form $\overline{\mu}_\gamma$ is $\partial_\theta$-invariant. Hence, $\mathfrak M^1_{\partial_\theta}$ sends $\mathcal P^{\overline{\mu}_\gamma}$ into itself.
\end{cor}
\begin{proof}
First, $\overline{\mu}_\gamma$ is $\partial_\theta$-invariant since $\Phi^{\partial_\theta}$ is a flow of isometries and thus $\mu_\gamma$ and $K$ are invariant under the flow. Second, if $\beta\in\mathcal P^{\overline{\mu}_\gamma}$, then the previous proposition implies that $d\big(\mathfrak M^1_{\partial_\theta}(\beta)\big)=\mathfrak M^2_{\partial_\theta}(d\beta)=\mathfrak M^2_{\partial_\theta}(\overline{\mu}_\gamma)=\overline{\mu}_\gamma$.
Hence, $\mathfrak M^1_{\partial_\theta}(\beta)\in\mathcal P^{\overline{\mu}_\gamma}$.
\end{proof}
Suppose now that $\beta$ lies in $\mathcal P^{\overline{\mu}_\gamma}\cap\Omega^1_{\partial_\theta}(S^2_\gamma)$. Thus, $\beta=\beta_t(t)dt+\beta_\theta(t)d\theta$. The function $\beta_\theta$ is uniquely defined by $d\beta=\overline{\mu}_\gamma=(1-K)\mu_\gamma$ and the fact that it vanishes on the boundary of $[0,\ell_\gamma]$. This stems from the equalities
\begin{equation}
d\beta=d\Big(\beta_tdt+\beta_\theta d\theta\Big)=\dot{\beta}_\theta dt\wedge d\theta=\frac{\dot{\beta}_\theta}{\gamma}\mu_\gamma
\end{equation}
Recalling the formula for $K$, we have $\dot{\beta}_\theta=\gamma+\ddot{\gamma}$. Hence, $\beta_\theta=\Gamma+\dot{\gamma}$, where $\Gamma:[0,\ell_\gamma]\rightarrow[-1,1]$ is the only primitive of $\gamma$ such that $\Gamma(0)=-1$. Notice that
\begin{enumerate}
 \item $\Gamma$ is increasing,
 \item $\Gamma(\ell_\gamma)=-1+\int_0^{\ell_\gamma}\gamma(t)dt=-1+2=1$,
 \item the odd derivatives of $\Gamma$ vanish at the boundary of its domain.
\end{enumerate}
Since $\beta_\theta$ and its derivatives of odd orders are zero for $t=0$ and $t=\ell_\gamma$, the $1$-form $\beta^\gamma:=\beta_\theta d\theta$ is well defined also at the poles and belongs to $\mathcal P^{\overline{\mu}_\gamma}\cap\Omega^1_{\partial_\theta}(S^2_\gamma)$. Finally, the norm of this new primitive is less than or equal to the norm of $\beta$:
\begin{equation}
|\beta_{(t,\theta)}|=\sqrt{\beta_t^2+\frac{\beta^2_\theta}{\gamma^2}}\geq \left|\frac{\beta_\theta}{\gamma}\right|=|\beta^\gamma_{(t,\theta)}|.
\end{equation}
Summing up, we have proven the following proposition.
\begin{prp}\label{estm}
There exists a unique $\beta^\gamma\in\mathcal P^{\overline{\mu}_\gamma}$ of the form $\beta^\gamma=\beta^\gamma_\theta(t)d\theta$. It satisfies $m_\gamma=\Vert \beta^\gamma\Vert$. Moreover $\beta^\gamma_\theta=\Gamma+\dot{\gamma}$ and
\begin{equation}
\Vert \beta^\gamma\Vert=\sup_{t\in[0,\ell_\gamma]}\left|\frac{\Gamma(t)+\dot{\gamma}(t)}{\gamma(t)}\right|.
\end{equation}
\end{prp}

Using the previous proposition, we can compute $m_\gamma$ directly from the function $\gamma$. As an application, we now produce a simple case where $m_\gamma$ can be bounded from above.
\begin{prp}\label{curvin}
Suppose $\gamma:[0,\ell_\gamma]\rightarrow\mathbb R$ is a normalized profile function such that $\gamma(t)=\gamma(\ell_\gamma-t)$. If $K$ is increasing in the variable $t$, for $t\in[0,\ell_\gamma/2]$, then $m_\gamma\leq1$ and $\mathcal C_\gamma=(0,+\infty)$.
\end{prp}
\begin{proof}
We observe that the functions $\Gamma$, $\dot{\gamma}$ and, hence, $\beta^\gamma_\theta$ are odd with respect to the point $\ell_\gamma/2$. This means, for example, that $\beta^\gamma_\theta(t)=-\beta^\gamma_\theta(\ell_\gamma-t)$. Therefore, in order to compute $m_\gamma$, we can restrict the attention to the interval $[0,\ell_\gamma/2]$. We know that $\beta^\gamma_\theta(0)=\beta^\gamma_\theta(\ell_\gamma/2)=0$ and we claim that, if $K$ is increasing, $\beta^\gamma_\theta$ is positive in the interior. This descends from the fact that $\dot{\beta}^\gamma_\theta=(1-K)\gamma$ passes from nonnegative to nonpositive at an interior critical point. Hence, such a point must be a local maximum and so $\beta^\gamma_\theta$ cannot assume negative values.

Let us estimate $\beta^\gamma_\theta/\gamma$ at an interior absolute maximizer $t_0$ (we can have more than one maximizer if $K=1$ on an open interval). The condition $\frac{d}{dt}\big|_{t=t_0}\frac{\beta^\gamma_\theta}{\gamma}=0$ is equivalent to
\begin{equation}\label{valcrit}
\frac{\big(\Gamma(t_0)+\dot{\gamma}(t_0)\big)\dot{\gamma}(t_0)}{\gamma^2(t_0)}=1-K(t_0).
\end{equation}
Since $\Gamma(t_0)+\dot{\gamma}(t_0)=\beta^\gamma_\theta(t_0)\geq0$ and $\dot{\gamma}(t_0)>0$, we see that $1-K(t_0)\geq0$. Moreover, using that $\Gamma(t_0)<0$, we get
\begin{equation}
\left(\frac{\dot{\gamma}(t_0)}{\gamma(t_0)}\right)^2>1-K(t_0).
\end{equation}
Finally, exploiting Equation \eqref{valcrit} again, we find
\begin{equation}
m_\gamma=\frac{\beta^\gamma_\theta(t_0)}{\gamma(t_0)}=\big(1-K(t_0)\big)\frac{\gamma(t_0)}{\dot{\gamma}(t_0)}\leq \sqrt{1-K(t_0)}\leq 1.
\end{equation}
The fact that $\mathcal C_\gamma=(0,+\infty)$ now follows from relation \eqref{conint1}.
\end{proof}
To complement the previous proposition we show that if, on the contrary, we assume that the curvature of $S^2_\gamma$ is sufficiently concentrated at one of the poles, $m_\gamma$ can be arbitrarily large. We are going to prove this behaviour in the restricted class of strictly convex surfaces since, as we explain later, we conjecture that the magnetic systems are of contact type at every energy level in this case. Before we need a preliminary lemma. Recall that $S^2_\gamma$ is convex, i.e.\ $K\geq0$, if and only if $\ddot{\gamma}(t)\leq0$.
\begin{lem}\label{bigmlem}
For every $0<\delta<\frac{\pi}{2}$ and for every $\varepsilon>0$ there exists a normalized profile function $\gamma_{\delta,\varepsilon}$ such that $S^2_{\gamma_{\delta,\varepsilon}}$ is convex and
\begin{equation}\label{inelem}
\dot{\gamma}_{\delta,\varepsilon}(\delta)<\varepsilon.
\end{equation}
\end{lem}
\begin{proof}
Given $\delta$ and $\varepsilon$, we find $a\in(\frac{2\delta}{\pi},1)$ such that $0<\sin\left(\frac{\delta}{a}\right)<\varepsilon$. This is achieved by taking $\frac{\delta}{a}$ very close to $\pi/2$ from below. Consider the profile function $\gamma_{a}:[-\frac{\pi a}{2},\frac{\pi a}{2}]\rightarrow[0,a]$ of a round sphere of radius $a$, where the domain is taken to be symmetric to zero to ease the following notation. It is defined by $\gamma_a(t)=a\cos(\frac{t}{a})$. Then, $\dot{\gamma}_a(-\frac{\pi a}{2}+\delta)=\cos\left(\frac{\delta}{a}\right)<\varepsilon$ and so, up to shifting the domain again, Inequality \eqref{inelem} is satisfied. However $\gamma_a$ is not normalized since
\begin{equation}
\int_{-\frac{\pi a}{2}}^{\frac{\pi a}{2}}\gamma_a(t)dt=2a^2<2. 
\end{equation}
In order to get the normalization in such a way that Inequality \eqref{inelem} is not spoiled, we are going to stretch the sphere in the interval $(-(\frac{\pi a}{2}-\delta),\frac{\pi a}{2}-\delta)$.

We claim that, for every $C>0$ there exists a diffeomorphism $F_C:\mathbb R\rightarrow\mathbb R$ with the property that
\begin{itemize}
 \item it is odd: $\forall t\in\mathbb R,\ F_C(t)=-F_C(-t)$;
 \item for $t\geq \frac{\pi a}{2}-\delta$, $F_C(t)=t+C$ and for $t\leq -(\frac{\pi a}{2}-\delta)$, $F_C(t)=t-C$;
 \item $\dot{F}_C\geq1$;
 \item for $t\leq0$, $\ddot{F}_C(t)\geq0$ and for $t\geq0$, $\ddot{F}_C(t)\leq0$. 
\end{itemize}
Such a map can be constructed as a time $C$ flow map $\Phi^\psi_C$, where $\psi:\mathbb R\rightarrow\mathbb R$ is an odd increasing function such that, for $t\geq \frac{\pi a}{2}-\delta$, $\psi(t)=1$.

Consider the function $\gamma^C_a:[-C-\frac{\pi a}{2},C+\frac{\pi a}{2}]\rightarrow\mathbb R$, where $\gamma^C_a(s):=\gamma_a(F^{-1}_C(s))$. One can check that $\gamma^C_a$ (up to a shift of the domain) is a profile function satisfying the convexity conditions and for which \eqref{inelem} holds. To finish the proof it is enough to find a positive real number $C_2$ such that $\int_\mathbb R\gamma^{C_2}_a=2$. Since we know that $\gamma^{0}_a=\gamma_a$ and $\int_\mathbb R\gamma_a<2$, it suffices to show that
\begin{equation}
\lim_{C\rightarrow+\infty}\int_\mathbb R\gamma^{C}_a(s)ds=+\infty.
\end{equation}
Observe that $b:=\gamma_a^C(-C-\frac{\pi a}{2}+\delta)=\gamma_a(-\frac{\pi a}{2}+\delta)=a\sin(\frac{\delta}{a})>0$. Then, for $s\in[-C-\frac{\pi a}{2}+\delta,C+\frac{\pi a}{2}-\delta]$, $\gamma_a^C(s)\geq b$ and we have the lower bound
\begin{equation}
\int_\mathbb R\gamma^{C}_a(s)ds\geq \int_{-C-\frac{\pi a}{2}+\delta}^{C+\frac{\pi a}{2}-\delta}\gamma^{C}_a(s)ds\geq\int_{-C-\frac{\pi a}{2}+\delta}^{C+\frac{\pi a}{2}-\delta}\!b\, ds=2(C+\frac{\pi a}{2}-\delta)b.
\end{equation}
The last quantity tends to infinity as $C$ tends to infinity.
\end{proof}
\begin{proof}[Proof of Proposition \ref{bigm}]
Fix an $\varepsilon_0<1$. Take any $\delta<\sqrt{1-\varepsilon_0}$ and consider the normalized profile function $\gamma_{\delta,\varepsilon_0}$ given by the lemma. We know that
\begin{equation}
\gamma_{\delta,\varepsilon_0}(\delta)=\int_0^\delta\dot{\gamma}_{\delta,\varepsilon_0}(t)dt\leq\int_0^\delta1\, dt=\delta. 
\end{equation}
In the same way we find $\Gamma_{\delta,\varepsilon_0}(\delta)\leq -1+\delta^2$. From these two inequalities we get
\begin{equation}
\Gamma_{\delta,\varepsilon_0}(\delta)+\dot{\gamma}_{\delta,\varepsilon_0}(\delta)<-1+\delta^2+\varepsilon_0<0.
\end{equation}
This yields the following lower bound for $m_{\gamma_{\delta,\varepsilon_0}}$:
\begin{equation}
m_{\gamma_{\delta,\varepsilon_0}}\geq\left|\frac{\Gamma_{\delta,\varepsilon_0}(\delta)+\dot{\gamma}_{\delta,\varepsilon_0}(\delta)}{\gamma_{\delta,\varepsilon_0}(\delta)}\right|\geq \frac{1-\varepsilon_0}{\delta}+\delta.
\end{equation}
The proposition is proven taking $\delta$ small enough.
\end{proof}
To sum up, we saw that the rotational symmetry gives us a good understanding of the set $\mathcal C_\gamma$. Understanding the set $\op{Con}_\gamma$ is more subtle. In the next subsection we perform this task only numerically and when $K_m>0$. As a first step, we will briefly study the symplectic reduction associated to the symmetry and the reduced dynamics in this case (for the general theory of symplectic reduction we refer to \cite{abm}). Proposition \ref{paract} and the numerical computation outlined in Section \ref{sub_act} suggest that, if $K_m>0$, the contact property holds. To complete the picture, we show in Proposition \ref{nonnec} that the assumption on the magnetic curvature is not necessary and in Proposition \ref{noncon} that there are cases where the magnetic curvature is not positive and that are not of contact type.

\subsection{The symplectic reduction}
As a first step we observe that the flow $\Phi^{\partial_\theta}$ lifts to a flow $d\Phi^{\partial_\theta}$ on $SS^2_\gamma$. Since $\Phi^{\partial_\theta}$ is a flow of isometries, $d\Phi^{\partial_\theta}_{\theta'}$ in coordinates is simply translation in the variable $\theta$:
$d\Phi^{\partial_\theta}_{\theta'}(t,\varphi,\theta)=(t,\varphi,\theta+\theta')$. Hence, $d\Phi^{\partial_\theta}$ is generated by $\widehat{\partial}_\theta$. As the flow $\Phi^{\widehat{\partial}_\theta}=d\Phi^{\partial_\theta}$ is $2\pi$-periodic and acts freely on $SS^2_\gamma$, we can take its quotient $\widehat{SS^2_\gamma}$ with respect to this $\mathbb R/2\pi\mathbb Z$-action. Furthermore, the quotient map $\widehat{\pi}:SS^2_\gamma\rightarrow \widehat{SS^2_\gamma}$ is a submersion. The variables $t$ and $\varphi$ descend to coordinates defined on $\widehat{SS^2_\gamma}$ minus two points, which are the fibers of the unit tangent bundle over the south and north pole. In these coordinates we simply have $\widehat{\pi}(t,\varphi,\theta)=(t,\varphi)$. In particular, $\widehat{SS^2_\gamma}$ is diffeomorphic to a $2$-sphere.

Any $\tau\in\Omega^k_{\widehat{\partial}_\theta}(SS^2_\gamma)$ such that $\imath_{\widehat{\partial}_\theta}\tau=0$ passes to the quotient and yields a well-defined form on $\widehat{SS^2_\gamma}$. The $2$-form $\imath_{\widehat{\partial}_\theta}\nu_\gamma$ falls into this class and, hence, there exists $\Theta_\gamma\in\Omega^2(\widehat{SS^2_\gamma})$ such that $\imath_{\widehat{\partial}_\theta}\nu_\gamma=\widehat{\pi}^*\Theta_\gamma$. Moreover, the form $\Theta_\gamma$ is symplectic on $\widehat{SS^2_\gamma}$. On the other hand, $X^m_\gamma$ is also $\widehat{\partial}_\theta$-invariant thanks to Equation \eqref{relaz1}. So there exists $\widehat{X}^m_\gamma\in\Gamma(\widehat{SS^2_\gamma})$ such that $d\widehat{\pi}(X^m_\gamma)=\widehat{X}^m_\gamma$. We claim that this new vector field is $\Theta_\gamma$-Hamiltonian. First, notice that if $\beta^\gamma=\beta_\theta d\theta$ is as defined in Proposition \ref{estm}, $\tau^{m,\beta^\gamma}\in\Omega^1_{\widehat{\partial}_\theta}(SS^2_\gamma)$. Second, using Cartan's identity, we have
\begin{equation}\label{chain}
\imath_{X^m_\gamma}(\imath_{\widehat{\partial}_\theta}\nu_\gamma)=-\imath_{\widehat{\partial}_\theta}(d\tau^{m,\beta^\gamma})=-\mathcal L_{\widehat{\partial}_\theta}\tau^{m,\beta^\gamma}+d\big(\imath_{\widehat{\partial}_\theta}\tau^{m,\beta^\gamma}\big)=d\big(\tau^{m,\beta^\gamma}(\widehat{\partial}_\theta)\big).
\end{equation}
Define $I_{m,\gamma}:=\tau^{m,\beta^\gamma}(\widehat{\partial}_\theta)$. Since $I_{m,\gamma}$ is $\widehat{\partial}_\theta$-invariant, there exists $\widehat{I}_{m,\gamma}:\widehat{SS^2_\gamma}\rightarrow\mathbb R$ such that $I_{m,\gamma}=\widehat{\pi}^*\widehat{I}_{m,\gamma}$. 
Thus, $I_{m,\gamma}$ is an integral of motion for $X^m_\gamma$ and, reducing the equality \eqref{chain} to $\widehat{SS^2_\gamma}$, we find that $\widehat{X}^m_\gamma$ is the $\Theta_\gamma$-Hamiltonian vector field generated by $-\widehat{I}_{m,\gamma}$.  

Using Equation \eqref{relaz2}, we find the coordinate expression
\begin{equation}
\widehat{I}_{m,\gamma}(t,\varphi)=m\gamma(t)\sin\varphi-\Gamma(t).
\end{equation}

Let us consider now the two auxiliary functions $\widehat{I}^\pm_{m,\gamma}:[0,\ell_\gamma]\rightarrow \R$ defined by $\widehat{I}^\pm_{m,\gamma}(t):=\widehat{I}_{m,\gamma}(t,\pm \pi/2)=\pm m\gamma(t)-\Gamma(t)$. We know that
\begin{equation}
\widehat{I}^-_{m,\gamma}(t)\leq \widehat{I}_{m,\gamma}(t,\varphi)\leq \widehat{I}^+_{m,\gamma}(t),
\end{equation}
with equalities if and only if $\varphi=\pm\pi/2$. On the one hand, we have $\widehat{I}^+_{m,\gamma}\geq-1$ and $\widehat{I}^+_{m,\gamma}(t)=-1$ if and only if $t=\ell_\gamma$. On the other hand, $\widehat{I}^+_{m,\gamma}$ attains its maximum in the interior. Indeed,
\begin{equation}\label{derI}
\frac{d}{dt}\widehat{I}^+_{m,\gamma}=m\dot{\gamma}-\gamma.
\end{equation}
and so $\frac{d}{dt}\widehat{I}^+_{m,\gamma}(0)=m>0$. Since $\widehat{I}^+_{m,\gamma}(0)=1$, the maximum is also strictly bigger than $1$. A similar argument tells us that the maximum of $\widehat{I}^-_{m,\gamma}$ is $1$ and it is attained at $0$ and the minimum of $\widehat{I}^-_{m,\gamma}$ is strictly less than $-1$ and it is attained in $(0,\ell_\gamma)$. 
As a consequence, $\max \widehat{I}_{m,\gamma}=\max \widehat{I}^+_{m,\gamma}>1$ and $\min \widehat{I}_{m,\gamma}=\min \widehat{I}^-_{m,\gamma}<-1$.

In the next proposition we deal with critical points of $\widehat{I}_{m,\gamma}$. In particular, we show that, if $K_m>0$, the only critical points are the maximizer and the minimizer (which are unique). In this case the dynamics of $\widehat{X}^m_\gamma$ is very simple: besides the two rest points, all other orbits are periodic and wind once in the complement of these two points.
\begin{prp}\label{reddyn}
If we denote by $\widehat{\op{Crit}}^\pm_{m,\gamma}$ the critical points of $\widehat{I}^\pm_{m,\gamma}$ and by $\widehat{\op{Crit}}_{m,\gamma}$ those of $\widehat{I}_{m,\gamma}$, we have
\begin{equation}\label{critsub}
\widehat{\op{Crit}}_{m,\gamma}=\widehat{\op{Crit}}^-_{m,\gamma}\times\left\{-\frac{\pi}{2}\right\}\bigcup\, \widehat{\op{Crit}}^+_{m,\gamma}\times\left\{+\frac{\pi}{2}\right\}.
\end{equation}
Moreover, $t\in \widehat{\op{Crit}}^\pm_{m,\gamma}$ if and only if
\begin{equation}\label{creq}
\pm m\dot{\gamma}(t)=\gamma(t).
\end{equation}
In this case $\{(t,\pm\pi/2,\theta)\ |\ \theta\in\mathbb R/2\pi\mathbb Z\}$ is the support of a closed orbit for $X^m_\gamma$. These are exactly the periodic orbits whose projection to $S^2_\gamma$ is a latitude.

Finally, if $K_m>0$, $\widehat{\op{Crit}}^\pm_{m,\gamma}$ contains only the absolute maximizer (respectively minimizer) of $\widehat{I}^\pm_{m,\gamma}$. We denote this unique element by $\widehat{t}^\pm_{m,\gamma}$. The complement of $\{(\widehat{t}^-_{m,\gamma},-\pi/2),(\widehat{t}^+_{m,\gamma},\pi/2)\}$ in $\widehat{SS^2_\gamma}$ is foliated by closed orbits of $\widehat{X}^m_\gamma$ and the complement of $\{(\widehat{t}^-_{m,\gamma},-\pi/2,\theta),(\widehat{t}^+_{m,\gamma},\pi/2,\theta)\}$ in $SS^2_\gamma$ is foliated by $X^m_\gamma$-invariant tori. 
\end{prp}

\begin{proof}
Identity \ref{critsub} follows from the fact that $\partial_\varphi \widehat{I}_{m,\gamma}=0$ if and only if $\varphi=\pm\pi/2$. Recalling Equation \eqref{derI} we see that \eqref{creq} is exactly the equation for the critical points of $\widehat{I}^\pm_{m,\gamma}$. The statement about the relation between closed orbits of $X^m_\gamma$ and latitudes follows from the fact that at a critical point of $\widehat{I}_{m,\gamma}$, $\widehat{X}^m_\gamma=0$. Hence, on its preimage $X^m_\gamma$ is parallel to $\widehat{\partial}_\theta$. By the implicit function theorem the regular level sets of $\widehat{I}_{m,\gamma}$ and $I_{m,\gamma}$ are closed submanifolds of codimension $1$. In the latter case they are tori since $X^m_\gamma$ is tangent to them and nowhere vanishing. 

We prove now uniqueness under the hypothesis on the curvature. We carry out the computations for $\widehat{I}^+_{m,\gamma}$ only. To prove that the absolute maximizer is the only critical point, we show that if $t_0$ is critical, the function is concave at $t_0$. Indeed,
\begin{eqnarray*}
\frac{d^2}{dt^2}\widehat{I}^+_{m,\gamma}(t_0)=m\ddot{\gamma}(t_0)-\dot{\gamma}(t_0)&=&m\gamma(t_0)\left(\frac{\ddot{\gamma}(t_0)}{\gamma(t_0)}-\frac{\dot{\gamma}(t_0)}{m\gamma(t_0)}\right)\\
&=&-m\gamma(t_0)\left(K(t_0)+\frac{1}{m^2}\right)<0.
\end{eqnarray*}
\end{proof}
The picture above shows qualitatively $\widehat{I}^-_{m,\gamma}$ and $\widehat{I}^+_{m,\gamma}$ when $K_m>0$.
\begin{figure}
\includegraphics[width=4in]{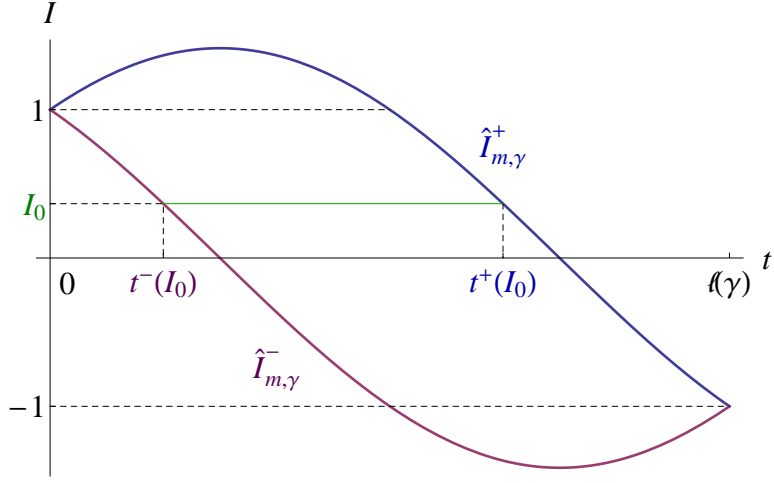}
\caption{Graphs of the functions $I^-_{m,\gamma}$ and $I^+_{m,\gamma}$}\label{pic}
\end{figure} 

In order to decide whether $\omega^\gamma_m$ is of contact type or not, the first thing to do is to compute the action of latitudes. We do this in the next proposition.
\begin{prp}\label{paract}
Take $t_0\in(0,\ell_\gamma)$ such that $\dot{\gamma}(t_0)\neq0$ and let $m_{t_0}:=\left|\frac{\gamma(t_0)}{\dot{\gamma}(t_0)}\right|$. The lift of the latitude curve $\{t=t_0\}$ parametrized by arc length and oriented by $\big(\op{sign}\dot{\gamma}(t_0)\big)\partial_\theta$ is the support of a periodic orbit for $X^{m_{t_0}}_\gamma$. We call the associated invariant probability measure $\zeta_{t_0}$. Its action is given by
\begin{equation}\label{acpar}
\mathcal A(\zeta_{t_0})=\frac{\gamma(t_0)^2-\dot{\gamma}(t_0)\Gamma(t_0)}{\dot{\gamma}(t_0)^2}
\end{equation}
and $I_{m_{t_0},\gamma}\big|_{\op{supp}\zeta_{t_0}}=\dot{\gamma}(t_0)\mathcal A(\zeta_{t_0})$. In particular, $\mathcal A(\zeta_{t_0})>0$ when $K_{m_{t_0}}>0$.
\end{prp}
\begin{proof}
The curve $s\mapsto\left(t_0,\op{sign}\dot{\gamma}(t_0)\frac{\pi}{2},\frac{m\op{sign}\dot{\gamma}(t_0)}{\gamma(t_0)}s\right)$ is a periodic orbit by the previous proposition. On the support of this orbit we have
\begin{equation}
m_{t_0}v_\theta=\left|\frac{\gamma(t_0)}{\dot{\gamma}(t_0)}\right|\frac{\op{sign}\dot{\gamma}(t_0)}{\gamma(t_0)}=\frac{1}{\dot{\gamma}(t_0)} 
\end{equation}
and, as a consequence,
\begin{equation}
\tau^\gamma_{m_{t_0}}(X^{m_{t_0}}_\gamma)=\frac{\gamma(t_0)^2}{\dot{\gamma}(t_0)^2}-\beta^\gamma_\theta\frac{1}{\dot{\gamma}(t_0)}+1=\frac{\gamma(t_0)^2-\dot{\gamma}(t_0)\Gamma(t_0)}{\dot{\gamma}(t_0)^2}.
\end{equation}
Since this is a constant, we get the identity \eqref{acpar} for the action.

The second identity is proved using the definition of $I_{m_{t_0},\gamma}$:
\begin{equation*}
I_{m_{t_0},\gamma}\big|_{\op{supp}\zeta_{t_0}}=\left|\frac{\gamma(t_0)}{\dot{\gamma}(t_0)}\right|\gamma(t_0)\op{sign}\dot{\gamma}(t_0)-\Gamma(t_0)=\dot{\gamma}(t_0)\frac{\gamma(t_0)^2-\dot{\gamma}(t_0)\Gamma(t_0)}{\dot{\gamma}(t_0)^2}.
\end{equation*}
Under the curvature assumption, $I_{m_{t_0},\gamma}$ is maximized or minimized at $\op{supp}\zeta_{t_0}$ according to the sign of $\dot{\gamma}(t_0)$. In both cases $I_{m_{t_0},\gamma}\big|_{\op{supp}\zeta_{t_0}}$ and $\dot{\gamma}(t_0)$ have the same sign. Hence, also the third statement is proved. 
\end{proof}
This proposition shows that, when $K_m>0$, the action of the periodic orbits that project to latitudes is not an obstruction for $\omega^\gamma_m$ to be of contact type. Thus, as we also discuss in the next subsection, one could conjecture that under this hypothesis $\omega^\gamma_m$ is of contact type. On the other hand, we claim that having $\inf K+1/m^2>0$ is not a necessary condition for the contact property to hold. For this purpose it is enough to exhibit a non-convex surface for which $m_\gamma<2$. This can be achieved as a consequence of the fact that the curvature depends on the second derivative of $\gamma$, whereas $m_\gamma$ depends only on the first derivative.

We can start from $S^2_{\gamma_0}$, the round sphere of radius $1$, and find a non-convex surface of revolution $S^2_\gamma$, which is $C^1$-close to the sphere and coincides with it around the poles. This implies that $m_\gamma=\sup_{t\in[0,\ell_\gamma]}\left|\frac{\Gamma(t)+\dot{\gamma}(t)}{\gamma(t)}\right|$ can be taken smaller than $2$ since it is close as we like to $m_{\gamma_0}=0$. Hence, every energy level of $(S^2_\gamma,g_\gamma,\mu_\gamma)$ is of contact type and the following proposition is proved.
\begin{prp}\label{nonnec}
The condition $K_m>0$ is not necessary for $\Sigma_{1/2m^2}$ to be of contact type.
\end{prp}

On the other hand, we now show that is not true, in general, that the contact property holds on every energy level.
\begin{prp}\label{noncon}
There exists a symplectic magnetic system $(S^2,g,\sigma)$ that has an energy level not of contact type.
\end{prp}
\begin{proof}
We will achieve this goal by finding $m$ and $\gamma$ such that $X^m_\gamma$ has a closed orbit projecting to a latitude with negative action. Then, the proof is complete applying Proposition \ref{mcdcri}. 

Fix some $\varepsilon\in(0,1)$. We claim that, for every $\delta\in(0,\pi/2)$, there exists a normalized profile function $\gamma_{\delta,\varepsilon}$ such that
\begin{equation}\label{strein}
\dot{\gamma}_{\delta,\varepsilon}(\delta)<-\varepsilon.
\end{equation}
Such profile function can be obtained as in Lemma \ref{bigmlem}. Take an $a\in\left(\frac{\delta}{\pi},\frac{\delta}{\pi/2}\right)$ (this is equivalent to $\delta\in\left(\frac{\pi a}{2},\pi a\right)$) such that $-\sin\left(\frac{\delta}{a}-\frac{\pi}{2}\right)<-\varepsilon$. Consider the profile function $\gamma_a:[-\pi a/2,\pi a/2]\rightarrow\mathbb R$ of a round sphere of radius $a$. Thanks to the last inequality, $\dot{\gamma}_a(\delta-\frac{\pi a}{2})<-\varepsilon$ and, therefore, $\gamma_a$ satisfies \eqref{strein} (up to a shift of the domain). Now we stretch an interval compactly supported in $(\delta-\frac{\pi a}{2}, \pi a/2)$ by a family of diffeomorphisms $F_C$ as we did in Lemma \ref{bigmlem} (here the condition on the second derivative of $F_C$ is not necessary). Thus, we obtain a family of profile functions $\gamma^C_a$ satisfying \eqref{strein}. Since the area diverges with $C$, we find $C_2>0$ such that $\gamma^{C_2}_a$ is normalized. This finishes the proof of the claim.

Given $\gamma_{\delta,\varepsilon}$ satisfying \eqref{strein}, we have that $\gamma_{\delta,\varepsilon}(\delta)\leq\delta$ and $\Gamma_{\delta,\varepsilon}(\delta)\leq-1+\delta^2$. The latitude at height $t=\delta$ of such surface is a closed orbit for $X^{m_{\delta,\varepsilon}}_{\gamma_{\delta,\varepsilon}}$, where $m_{\delta,\varepsilon}:=\frac{\gamma_{\delta,\varepsilon}(\delta)}{|\dot{\gamma}_{\delta,\varepsilon}(\delta)|}$. Using formula \eqref{acpar} we see that the action of the corresponding invariant measure $\zeta_{\delta,\varepsilon}$ is negative for $\delta$ small enough:
\begin{equation}
\mathcal A(\zeta_{\delta,\varepsilon})\leq\frac{\delta^2-(-\varepsilon)(-1+\delta^2)}{\varepsilon^2}=-\frac{1}{\varepsilon}+o(\delta)<0.
\end{equation}
\end{proof}

\subsection{Action of ergodic measures}\label{sub_act}
When $K_m>0$, we also have a way to compute numerically the action of ergodic invariant measures. We consider only ergodic measures since they are the extremal points of the set of probability invariant measures by Choquet's Theorem and, therefore, it is enough to check the positivity of the action of these measures, in order to apply Proposition \ref{mcdcri}. Every ergodic measure $\zeta$ is concentrated on a unique level set $\{I_{m,\gamma}=I(\zeta)\}$, for some $I(\zeta)\in\R$. Moreover, if $I(\zeta)=I(\zeta')$ there exists a rotation $\Phi^{\widehat{\partial}_\theta}_{\theta'}$ such that $\big(\Phi^{\widehat{\partial}_\theta}_{\theta'}\big)_*\zeta=\zeta'$. Since the action is $\widehat{\partial}_\theta$-invariant, we deduce that the action is a function of $I(\zeta)$ only and we can define $\mathcal A:[\min I_{m,\gamma},\max I_{m,\gamma}]\rightarrow\R$. We already have an expression for the action at the minimum and maximum of $I_{m,\gamma}$. We now give a formula for the action when $I\in(\min I_{m,\gamma},\max I_{m,\gamma})$.

Every integral line $z:\R\rightarrow SS^2$ of $X^m_\gamma$, such that $I_{m,\gamma}(z)=I$ oscillates between the latitudes at height $t^-(I)$ and $t^+(I)$. Their numerical values can be easily read off from the graphs of $\widehat{I}^\pm_{m,\gamma}$ drawn in Figure \ref{pic}. If we take $z$ with $t(z(0))=t^-(I)$,  there exists a smallest $s(I)>0$ such that $t\big(z(s(I))\big)=t^+(I)$. By Birkhoff's ergodic theorem
\begin{equation}
\mathcal A(I)=\frac{1}{s(I)}\int_0^{s(I)}\left(m^2-m\frac{\beta_\theta(t)\sin\varphi}{\gamma(t)}(z(s))+1\right)ds.
\end{equation}
Using this identity, we computed with \textit{Mathematica} the action for ellipsoids of revolution and we found that is positive on every energy level. This shows numerically that $\op{Con}(g,\sigma)=(0,+\infty)$ for these systems by making use of Proposition \ref{mcdcri}, hence corroborating Conjecture \ref{conj}. On the other hand, we know that, when the ellipsoid is very thin, its curvature is concentrated on its poles and hence, by Proposition \ref{bigm}, the set $\mathcal C(g,\sigma)$ is not the whole $(0,+\infty)$. Therefore, these data would also show numerically that, in general, the inclusion $\mathcal C(g,\sigma)\subset\op{Con}(g,\sigma)$ is strict.
\section{Periodic orbits and double covers}\label{sec_dyn1}
The aim of this section is to establish Proposition \ref{covthe}. Consider $\R\Pro^3$ as the quotient of $S^3$ by the antipodal map $A:S^3\rightarrow S^3$. The quotient map $p:S^3\rightarrow \R\Pro^3$ is a double cover with the map $A$ as the only non-trivial deck transformation. There is a bijection $Z\mapsto \widehat{Z}$ between $\Gamma(\R\Pro^3)$ and $\Gamma_A(S^3)\subset\Gamma(S^3)$ the subset of $A$-invariant vector fields. The antipodal map permutes the flow lines of $\widehat{Z}$. Moreover, a lift of a trajectory for $Z$ is a trajectory for $\widehat{Z}$ and the projection of a trajectory for $\widehat{Z}$ is a trajectory for $Z$. In the next lemma we restrict this correspondence to prime contractible periodic orbits of $Z$. 
\begin{lem}\label{lemlin}
There is a bijection between contractible prime orbits $z$ of $Z$ and pairs of antipodal prime orbits $\{\widehat{z},A(\widehat{z})\}$ of $\widehat{Z}$ such that $\widehat{z}$ and $A(\widehat{z})$ are disjoint. Furthermore, the linking number $lk(\widehat{z},A(\widehat{z}))$ between them is even.
\end{lem}
\begin{proof} 
Associate to a contractible periodic orbit $z$ its two lifts $\widehat{z}_1$ and $\widehat{z}_2=A(\widehat{z}_1)$. Since $z$ is contractible both lifts are closed. They are also prime since a lift of an embedded path is still embedded. Suppose that the two lifts intersect. This implies that there exist points $t_1$ and $t_2$ such that $\widehat{z}_1(t_1)=\widehat{z}_2(t_2)$. Applying $p$ to this equality, we find $z(t_1)=z(t_2)$ and so $t_1=t_2$ modulo the period of $z$. Hence, $\widehat{z}_1=\widehat{z}_2$ contradicting the fact that the two lifts are distinct.

For the inverse correspondence, associate to two antipodal disjoint prime periodic orbits $\{\widehat{z},A(\widehat{z})\}$ their common projection $p(\widehat{z})$. The projected curve is contractible since its lifts are closed. Moreover, it is prime since, if $p(\widehat{z})(t_1)=p(\widehat{z})(t_2)$, either $\widehat{z}(t_1)=A(\widehat{z})(t_2)$ and $\widehat{z}$ and $A(\widehat{z})$ are not disjoint or $\widehat{z}(t_1)=\widehat{z}(t_2)$ and $t_1=t_2$ modulo the period of $\widehat{z}$.

We now compute the linking number between the two knots. Consider $S^3$ as the boundary of $B^4$ the unit ball inside $\mathbb R^4$ and denote still by $A$ the antipodal map on $B^4$, which extends the antipodal map on $S^3$. Take an embedded surface $S_1\subset B^4$ such that $\partial S_1=\widehat{z}_1$ and transverse to the boundary of $B^4$. By a small perturbation we can also assume that $0\in B^4$ does not belong to the surface. The antipodal surface $S_2:=A(S_1)$ has the curve $\widehat{z}_2$ as boundary and $lk(\widehat{z}_1,\widehat{z}_2)$ is equal to the intersection number between $S_1$ and $S_2$. By perturbing again $S_1$ we can suppose that all the intersections are transverse. This follows from the fact that, if we change $S_1$ close to a point $z$ of intersection, this will affect $S_2=A(S_1)$ only near the antipodal point $A(z)=-z$, which is different from $z$ since the origin does not belong to $S_1$. Now that transversality is achieved, we claim that the number of intersections is even. This stems from the fact that, if $z\in S_1\cap S_2$, then $A(z)\in A(S_1)\cap A(S_2)=S_2\cap S_1$ and $z$ and $A(z)$ are different since $z\neq0$. This implies that the intersection number between the two surfaces is even as well and the lemma is proved. We also notice that the sign of the intersection at $z$ is the same as the sign at $A(z)$, since $A$ preserves the orientation. Thus, we cannot conclude that the total intersection number is zero and indeed for any $k\in\mathbb Z$ one can find a pair of antipodal knots, whose linking number is $2k$.
\end{proof}

\begin{proof}[Proof of Proposition \ref{covthe}]
We can suppose without loss of generality, that $N=\R\Pro^3$ and that $p$ is the quotient covering map. By Corollary \ref{hwzcor} there exist two prime closed orbits $\widehat{z}_1$ and $\widehat{z}_2$ of $\widehat{Z}$ forming a Hopf link and if there is any other periodic orbit geometrically distinct from these two, $\widehat{Z}$ has infinitely many periodic orbits.

We claim that $z_1:=p(\widehat{z}_1)$ and $z_2:=p(\widehat{z}_2)$ are geometrically distinct closed orbits for $Z$ on $\R\Pro^3$. If, by contradiction, $z_1$ coincides with $z_2$, by Lemma \ref{lemlin}, $\widehat{z}_1$ and $\widehat{z}_2$ are antipodal and their linking number is even. This is a contradiction since $|lk(\widehat{z}_1,\widehat{z}_2)|=1$. Therefore, we conclude that $z_1$ and $z_2$ are distinct. On the other hand, if $\widehat{Z}$ has infinitely many periodic orbits the same is true for $Z$. Hence, also $Z$ has either $2$ or infinitely many distinct periodic orbits.

If $Z$ has a prime contractible periodic orbit $w$, its lifts $\widehat{w}_1$ and $\widehat{w}_2$ are disjoint, antipodal and prime periodic orbits for $\widehat{Z}$ by Lemma \ref{lemlin}. Since $lk(\widehat{w}_1,\widehat{w}_2)$ is even, $\{\widehat{w}_1,\widehat{w}_2\}\neq \{\widehat{z}_1,\widehat{z}_2\}$
and, therefore, there are at least three distinct periodic orbits for $\widehat{Z}$. So there are infinitely many periodic orbits for $\widehat{Z}$ and, hence, also for $Z$.

The statement about contact forms is a consequence of Theorem \ref{hwzthm} and the relation $\widehat{R^\tau}=R^{p^*\tau}$.
\end{proof}

\section{Dynamical convexity and low energy values}\label{sec_dyn2}
In this last section we present two independent proofs of Proposition \ref{corfinalo}, which allows us to apply Proposition \ref{covthe} and, finally, get Theorem \ref{mainthm} about the existence of periodic orbits on low energy levels. As a common step we fix some $\beta\in\mathcal P^{\overline{\sigma}}$ and take an $m_\beta>0$ such that $(SS^2,\tau^{m,\beta})$ is a contact manifold for all $m\in[0,m_\beta)$. This is equivalent to asking $h_{m,\beta}>0$ for $m$ in $[0,m_\beta)$. We denote the Reeb vector field of $\tau^{m,\beta}$ by $R^{m,\beta}:=R^{\tau^{m,\beta}}=\frac{1}{h_{m,\beta}}X^m$. 

\subsection{Contactomorphism with a convex hypersurface}\label{conve}
The first strategy of proof relies on the following result, which immediately implies Proposition \ref{corfinalo}.
\begin{prp}\label{convi}
If $m\in[0,m_\beta)$, there exists a double cover $p_{m,\beta}:S^3\rightarrow SS^2$ and an embedding $\upsilon_{m,\beta}:S^3\rightarrow \mathbb C^2$ bounding a region starshaped around the origin and such that $p_{m,\beta}^*\tau^{m,\beta}=-\upsilon_{m,\beta}^*\lambda_{\op{st}}$. Furthermore, as $m$ goes to zero, $\upsilon_{m,\beta}$ tends in the $C^2$-topology to the embedding of $S^3$ as the Euclidean sphere of radius $\sqrt{2}$. In particular, $\upsilon_{m,\beta}$ is a convex embedding for $m$ sufficiently small.
\end{prp}
\begin{proof}[First proof of Proposition \ref{corfinalo}]
As proven in \cite[Theorem 3.7]{hwz1}, the contact forms of the type $\widetilde{\upsilon}^*\lambda_{\op{st}}$, with $\widetilde{\upsilon}:S^3\rightarrow\mathbb C^2$ a convex embedding, are tight and dynamically convex.
\end{proof}

We construct the double cover $p_{m,\beta}$ in three steps. 
\begin{lem}\label{lemconb}
There exists a diffeomorphism $F_{m,\beta}:SS^2\rightarrow SS^2$ and a function $q_{m,\beta}:SS^2\rightarrow\mathbb R$ such that 
\begin{equation}\label{gra}
F_{m,\beta}^*\tau^{m,\beta}=e^{q_{m,\beta}}\tau^{0,\beta}.
\end{equation}
The function $q_{m,\beta}$ tends to $q_{0,\beta}=0$ in the $C^2$-topology.
\end{lem}
\begin{proof}
We apply Gray's Stability Theorem to the family $m\mapsto\tau^{m,\beta}$ and get $F_{m,\beta}$ and $q_{m,\beta}$ satisfying \eqref{gra}. In particular, $q_{m,\beta}$ is obtained integrating in the variable $m$ the equation
\begin{equation}
\frac{d}{dm}q_{m,\beta}(z)=\left(\frac{d}{dm}\tau^{m,\beta}\right)(R^{m,\beta})_{F_{m,\beta}(z)}
\end{equation}
with the boundary condition $q_{0,\beta}(z)=0$. Since the function $(m,z)\mapsto\frac{d}{dm}q_{m,\beta}(z)$ is smooth (hence $C^2$) on $[0,m_\beta)\times SS^2$, the same is true for $(m,z)\mapsto q_{m,\beta}(z)$. Therefore, the map $m\mapsto q_{m,\beta}$ is continuous in the $C^2$-topology.
\end{proof}
Let $(S^2,g_0,\mu_0)$ be the magnetic system on the round sphere of radius $1$ given by the area form. We denote by $SS^2_0$ the unit sphere bundle, by $\jmath_0$ the rotation by $\pi/2$ and by $\psi_0$ the vertical form associated with the metric $g_0$. Our next task is to relate $\tau^{0,\beta}$ with $\psi_0$. For this purpose, we need the following proposition due to Weinstein \cite{wei3}. For a proof we refer to \cite[Appendix B]{gui}.
\begin{prp}
Suppose $E_i\rightarrow S^2$, with $i=0,1$, are two $S^1$-bundles endowed with $S^1$-connection forms $\tau_i\in\Omega^1(E_i)$. Call $\sigma_i\in \Omega^2(S^2)$ their curvature forms and suppose they are both symplectic and such that 
\begin{equation}\label{intsi}
\left|\int_{S^2}\sigma_0\right|=\left|\int_{S^2}\sigma_1\right|,
\end{equation}
Then, there is an $S^1$-equivariant diffeomorphism $B:E_0\rightarrow E_1$ such that $B^*\tau_1=\tau_0$.
\end{prp}
\begin{cor}\label{corconb}
There exists an $S^1$-equivariant diffeomorphism $B_\beta:SS^2_0\rightarrow SS^2$ such that $B_\beta^*\tau^{0,\beta}=\psi_0$.
\end{cor}
\begin{proof}
We show first that $\tau^{0,\beta}$ is an $S^1$-connection form. If $V$ is the vertical vector field, we have to check that
\begin{equation}
\bullet\ \tau^{0,\beta}(V)=1,\quad\quad\bullet\ \mathcal L_V\tau^{0,\beta}=0.
\end{equation}
Using Cartan's identity for the second equation we see that these requirements are equivalent to saying that $V$ is the Reeb vector field of the contact form $\tau^{0,\beta}$ and, hence, they are satisfied. Since $d\tau^{0,\beta}=-\pi^*\sigma$, we also know that the curvature form of the connection associated to $\tau^{0,\beta}$ is exactly $-\sigma$. By normalization \eqref{normi}, condition \eqref{intsi} is met and the previous proposition can be applied to get $B_\beta$.
\end{proof}
What we have found so far tells us that we only need to study the pullback of $\psi_0$ to $S^3$. This will be our next task. The ideas that we use come essentially from \cite{conoli} and \cite{pathar}.

Identify $\mathbb C^2$ with the space of quaternions by setting $\mathbf1:=(1,0)$, $\mathbf i:=(i,0)$, $\mathbf j:=(0,1)$ and $\mathbf k:=(0,i)$. With this choice left multiplication by $\mathbf i$ corresponds to the action of $J_{\op{st}}$. Let $\upsilon:S^3\rightarrow \mathbb C^2$ be the inclusion of the unit Euclidean sphere. Identify the Euclidean space $\mathbb R^3$ with the vector space spanned by $\mathbf i,\mathbf j,\mathbf k$ endowed with the restricted inner product. We think the round sphere $(S^2,g_0)$ as embedded in this version of the Euclidean space. Thus, the unit sphere bundle $SS^2_0$ is embedded in $\mathbb R^3\times\mathbb R^3$ as the pair of vectors $(u_1,u_2)$ such that $u_1,u_2\in S^2$ and $g_{\op{st}}(u_1,u_2)=0$.

Under this identification of the sphere bundle, if $z=(u_1,u_2)\in SS^2_0\subset\mathbb R^3\times\mathbb R^3$ and $Z=(v_1,v_2)\in T_zSS^2_0\subset\mathbb R^3\times\mathbb R^3$, we have that
\begin{equation}\label{formucon}
(\psi_0)_{z}(Z)=g_0\Big(v_2-g_{\op{st}}\big(v_2,u_1\big)u_1,{\jmath_0}_{u_1}(u_2)\Big)=g_{\op{st}}\big(v_2,{\jmath_0}_{u_1}(u_2)\big)
\end{equation}
as a consequence of the relation between the Levi Civita connections on $S^2$ and $\mathbb R^3$.

For any $U\in S^3$, we define a map $C_U:\mathbb R^3\hookrightarrow\mathbb R^3$ using quaternionic multiplication and inverse by $C_U(U')=U^{-1}U'U$. The quaternionic commutation relations and the compatibility between the metric and the multiplication tell us that $C_U$ restricts to an isometry of $S^2$. Hence, its differential $dC_U$ yields a diffeomorphism of the unit sphere bundle onto itself given by $(u_1,u_2)\mapsto d_{u_1}C_U(u_2)=(C_U(u_1),C_U(u_2))$. Moreover, since $C_U$ is an isometry, $(dC_U)^*\psi_0=\psi_0$. 

We are now ready to define the covering map $p_0:S^3\rightarrow SS^2_0$. It is given by $p_0(U):=d_{\mathbf i}C_U(\mathbf j)$. Let us compute the pull-back of $\psi_0$ by $p_0$.
\begin{prp}\label{stalif}
The covering map $p_0$ relates $\psi_0$ and $\lambda_{\op{st}}$ in the following way:
\begin{equation}\label{equicon}
p_0^*\psi_0=-2\upsilon^*\lambda_{\op{st}}. 
\end{equation}
\end{prp}
\begin{proof}
First of all we prove that both sides of \eqref{equicon} are invariant under right multiplication. For every $U\in S^3$, we define $R_U:S^3\rightarrow S^3$ as $R_U(U'):=U'U$. Thus, the identity $p_0\circ R_U=dC_U\circ p_0$ holds. Let us show that $p_0^*\psi_0$ is right invariant:
\begin{equation}
R_U^*\big(p_0^*\psi_0\big)=(p_0\circ R_U)^*\psi_0=(dC_U\circ p_0)^*\psi_0=p_0^*\big((dC_U)^*(\psi_0)\big)=p_0^*\psi_0.
\end{equation}
On the other hand, $\upsilon^*\lambda_{\op{st}}$ is also right invariant:
\begin{eqnarray*}
\big(R_U^*(\upsilon^*\lambda_{\op{st}})\big)_{U'}(W)=\big(\upsilon^*\lambda_{\op{st}}\big)_{R_U(U')}(dR_UW)&=&g_{\op{st}}\big((\mathbf iU')U,WU\big)\\
&=&g_{\op{st}}\big(\mathbf iU',W\big)\\
&=&\big(\upsilon^*\lambda_{\op{st}}\big)_{U'}(W),
\end{eqnarray*}
where we used that $R_U:\mathbb C^2\rightarrow\mathbb C^2$ is an isometry.

Therefore, it is enough to check equality \eqref{equicon} only at the point $\mathbf 1$. A generic element $W$ of $T_{\mathbf 1}S^3$ can be written as $s\mathbf i+w\mathbf j=s\mathbf i+\mathbf j\overline{w}$, where $w:=w_1\mathbf 1+w_2\mathbf i$ and $\overline{w}:=w_1\mathbf 1-w_2\mathbf i$ with $s,w_1 $ and $w_2$ real numbers. On the one hand,
\begin{equation}\label{ess1}
(\upsilon^*\lambda_{\op{st}})_{\mathbf 1}(W)=g_{\op{st}}\big(\mathbf i\mathbf 1,W\big)=g_{\op{st}}\big(\mathbf i,s\mathbf i+w\mathbf j\big)=s.
\end{equation}
On the other hand, we have that $d_{\mathbf 1}p_0(W)=(\mathbf iW-W\mathbf i,\mathbf jW-W\mathbf j)$. From the definition of $\psi_0$ we see that we are only interested in the second component:
\begin{equation}
\mathbf jW-W\mathbf j=\mathbf j(s\mathbf i+\mathbf j\overline{w})-(s\mathbf i+w\mathbf j)\mathbf j=-2s\mathbf k+(w-\overline{w})=-2s\mathbf k+2w_2\mathbf i.
\end{equation}
Now we apply formula \eqref{formucon} with $(u_1,u_2)=(\mathbf i,\mathbf j)$ and $v_2=-2s\mathbf k+2w_2\mathbf i$. In this case ${\jmath_0}_{u_1}$ is left multiplication by $\mathbf i$, so that ${\jmath_0}_{u_1}(u_2)=\mathbf i\mathbf j=\mathbf k$ and we find that
\begin{equation}\label{ess2}
g_{\op{st}}\big(-2s\mathbf k+2w_2\mathbf i,\mathbf k\big)=-2s.
\end{equation}
Comparing \eqref{ess1} with \eqref{ess2} we finally get $(p_0^*\psi_0)_{\mathbf 1}=-2(\upsilon^*\lambda_{\op{st}})_{\mathbf 1}$.
\end{proof}
Putting things together, we arrive at the following intermediate step. 
\begin{prp}\label{intermi}
There exists a covering map $p_{m,\beta}:S^3\rightarrow SS^2$ and a real function $\widehat{q}_{m,\beta}:S^3\rightarrow \mathbb R$ such that
\begin{equation}
p_{m,\beta}^*\tau^{m,\beta}=-2e^{\widehat{q}_{m,\beta}}\upsilon^*\lambda_{\op{st}}.
\end{equation}
Moreover the function $\widehat{q}_{m,\beta}$ tends to $0$ in the $C^2$-topology as $m$ goes to zero.
\end{prp}
\begin{proof}
Lemma \ref{lemconb} gives us $F_{m,\beta}:SS^2\rightarrow SS^2$ and $q_{m,\beta}:SS^2\rightarrow \mathbb R$. Corollary \ref{corconb} gives us $B_\beta:SS^2_0\rightarrow SS^2$. If we set $p_{m,\beta}:=F_{m,\beta}\circ B_\beta\circ p_0:S^3\rightarrow SS^2$, then
\begin{eqnarray*}
p_{m,\beta}^*\tau^{m,\beta}=p_0^*\big(B_\beta^*(F_{m,\beta}^*\tau^{m,\beta})\big)&=&p_0^*\big(B_\beta^*(e^{q_{m,\beta}}\tau^{0,\beta})\big)\\
&=&p_0^*(e^{q_{m,\beta}\circ B_\beta}\psi_0)\\
&=&-2e^{q_{m,\beta}\circ B_\beta\circ p_0}\upsilon^*\lambda_{\op{st}}.
\end{eqnarray*}
Defining $\widehat{q}_{m,\beta}:=q_{m,\beta}\circ B_\beta\circ p_0$ we only need to show that $\widehat{q}_{m,\beta}$ goes to $0$ in the $C^2$ topology. This is true since, by Lemma \ref{lemconb}, the same holds for $q_{m,\beta}$. 
\end{proof}

The final step in the proof of Proposition \ref{convi} is to notice that contact forms of the type $\rho\upsilon^*\lambda_{\op{st}}\in\Omega^1(S^3)$ with $\rho:S^3\rightarrow(0,+\infty)$ arise from embeddings of $S^3$ in $\mathbb C^2$ as the boundary of a star-shaped domain. In order to see this, define $\upsilon_{\sqrt{\rho}}:S^3\hookrightarrow \C^2$ as $\upsilon_{\sqrt{\rho}}(z):=\sqrt{\rho(z)}\upsilon(z)$. A computation shows that  $\upsilon_{\sqrt{\rho}}^*\lambda_{\op{st}}=\rho\upsilon^*\lambda_{\op{st}}$. If we define the function $Q_{\rho}:\mathbb C^2\rightarrow[0,+\infty)$ by $Q_{\rho}(z):=|z|^2/\rho(\frac{z}{|z|})$, then $\upsilon_{\sqrt{\rho}}(S^3)=\{Q_{\rho}=1\}$. As a consequence, $\upsilon_{\sqrt{\rho}}(S^3)$ is convex if and only if the Hessian of $Q_{\rho}$ is positive definite.

\begin{proof}[Proof of Proposition \ref{convi}]
Using the observation developed in the last paragraph we see that $\widehat{p}_{m,\beta}^*\tau^{m,\beta}=-\upsilon_{m,\beta}^*\lambda_{\op{st}}$ with $\upsilon_{m,\beta}:=\upsilon_{\sqrt{\rho_{m,\beta}}}$ and $\rho_{m,\beta}:=2e^{\widehat{q}_{m,\beta}}$. For small $m$, $\rho_{m,\beta}$ is $C^2$-close to the constant $2$ and, therefore, a direct computation shows that $Q_{\rho_{m,\beta}}$ has positive definite Hessian. Hence, the embedding $\upsilon_{m,\beta}$ is convex and the proof is complete.
\end{proof}

\subsection{A direct estimate of the index}\label{sub_ies}
In this subsection we are going to present an alternative proof of Proposition \ref{corfinalo}. The core of this approach is to give a direct proof of the dynamical convexity of $\tau^{m,\beta}$. After writing this note, we found out that an analogous argument is used in \cite[Section 3.2]{hrysal1}. 

Let $\Phi^{m,\beta}$ be the flow of $R^{m,\beta}$ and $\xi^{m,\beta}:=\ker\tau^{m,\beta}$ be the contact structure associated to $\tau^{m,\beta}$. The form $\omega_m=md\alpha-\pi^*\sigma$ is symplectic when restricted to $\xi^{m,\beta}$. Since $H^2(\R\Pro^3,\Z)=0$, $c_1(\xi^{m,\beta})=0$ and the Conley-Zehnder index of contractible periodic orbits of $R^{m,\beta}$ is well-defined. In the next lemma we exhibit an explicit global trivialization of $\xi^{m,\beta}$.

\begin{lem}
The contact structure $\xi^{m,\beta}$ admits a global $\omega_m$-symplectic frame:
\begin{equation}
\left\{ \begin{array}{rcl}
{\displaystyle H^{m,\beta}}&:=&\displaystyle\frac{\widetilde{H}^{m,\beta}}{\sqrt{h^{m,\beta}}}\vspace{.2cm}\\
{\displaystyle X^{m,\beta}}&:=&\displaystyle\frac{\widetilde{X}^{m,\beta}}{\sqrt{h^{m,\beta}}}	
\end{array}\right.,
\quad \mbox{where}\quad
\left\{\begin{array}{rcl}
{\displaystyle \widetilde{H}^{m,\beta}}&:=&\displaystyle H+\beta_x\big(\jmath_x(v)\big)V, \vspace{.2cm}\\
{\displaystyle \widetilde{X}^{m,\beta}}&:=&\displaystyle X+\big(\beta_x(v)-m\big)V.	
\end{array}\right.
\end{equation}
Call $\chi^{m,\beta}:\xi^{m,\beta}\rightarrow(\epsilon_{SS^2},\omega_0)$ the symplectic trivialization associated to this frame. It is given by
$\chi^{m,\beta}(Z)=\sqrt{h^{m,\beta}}(\eta(Z),\alpha(Z))\in\R^2$.
\end{lem}
\begin{proof}
To find a basis for $\xi^{m,\beta}$, we set $\widetilde{H}^{m,\beta}:=H+a_HV$ and $\widetilde{X}^{m,\beta}:=X+a_XV$, for some $a_H,a_X\in\mathbb R$. Imposing $\tau^{m,\beta}(\widetilde{H}^{m,\beta})=0$, we get 
\begin{equation*}
0=m\alpha(H+a_HV)-\pi^*\beta(H+a_HV)+\psi(H+a_HV)=0-\beta_x\big(\jmath_x(v)\big)+a_H.
\end{equation*}
Hence, we have $a_H=\beta_x\big(\jmath_x(v)\big)$. In the same way we find $a_X=\beta_x(v)-m$. In order to modify this basis into a symplectic one, we compute
\begin{eqnarray*}
\omega_m(\widetilde{H}^{m,\beta},\widetilde{X}^{m,\beta})&=&\omega_m(H+a_HV,X+a_XV)\\
&=&\omega_m(H,X)+a_X\omega_m(H,V)+a_H\omega_m(V,X)\\
&=&-(-f)+a_X\cdot(-m)+a_H\cdot0\\
&=&h_{m,\beta}.
\end{eqnarray*}
Thus $\left(H^{m,\beta},X^{m,\beta}\right)$, as defined in the statement of this lemma, is a symplectic basis. To find the coordinates of a vector $Z=a^1H^{m,\beta}+a^2X^{m,\beta}$ with respect to this basis, we notice that
\begin{equation*}
\eta(Z)=\eta(a^1H^{m,\beta}+a^2X^{m,\beta})=a^1\eta(H^{m,\beta})+a^2\eta(X^{m,\beta})=\frac{a^1}{\sqrt{h_{m,\beta}}}+a^2\cdot0.
\end{equation*}
In the same way, $\alpha(Z)=\frac{a^2}{\sqrt{h_{m,\beta}}}$, so that $(a^1,a^2)=\sqrt{h_{m,\beta}}(\eta(Z),\alpha(Z))$.
\end{proof}

To compute the index, we consider for each $z=(x,v)\in SS^2$ the path of symplectic matrices
\begin{equation}
\Psi^{m,\beta}_z(t):=\chi^{m,\beta}_{\Phi^{m,\beta}_t(z)}\circ d_z\Phi^{m,\beta}_t\circ \left(\chi^{m,\beta}_z\right)^{-1}\in \operatorname{Sp}(1).
\end{equation}
We define the auxiliary path $B^{m,\beta}_z(t):=\dot{\Psi}^{m,\beta}_z(\Psi^{m,\beta}_z)^{-1}\in\mathfrak{gl}(2,\R)$. The bracket relations \eqref{bra} for $(X,V,H)$ allow us to give the following estimate for this path.
\begin{lem}\label{linea}
We can write $B^{m,\beta}_z=J_{\op{st}}+\rho^{m,\beta}_z$, where $\rho^{m,\beta}_z:\mathbb R\rightarrow \mathfrak{gl}(2,\R)$ is a path of matrices, whose supremum norm is of order $O(m)$ uniformly in $z$ as $m$ goes to zero. In other words, there exist $\overline{m}>0$ and $C>0$ not depending on $z$, but only on $\sup f$, $\inf f$, $\Vert df\Vert$ and $\Vert\beta\Vert$, such that
\begin{equation}\label{odim}
\forall m\leq \overline{m},\quad\Vert\rho^{m,\beta}_z\Vert\leq Cm.
\end{equation}
\end{lem}
\begin{proof}
For any point $z\in SS^2$ and for any $\overrightarrow{a_0}=(a^1_0,a^2_0)\in\mathbb R^2$ we have a path $\overrightarrow{a}=(a^1,a^2):\mathbb R\rightarrow\mathbb R^2$ defined by the relation
\begin{equation}\label{lintri}
\overrightarrow{a}=\Psi^{m,\beta}_z\overrightarrow{a}_0.
\end{equation}
Using the definition of $\Psi^{m,\beta}_z$, we see that $\overrightarrow{a}$ satisfies
\begin{equation}
Z^{\overrightarrow{a_0}}_z(t):=d_z\Phi^{m,\beta}_t\left(a^1_0H^{m,\beta}_z+a^2_0X^{m,\beta}_z\right)=a^1(t)H^{m,\beta}_{\Phi^{m,\beta}_t(z)}+a^2(t)X^{m,\beta}_{\Phi^{m,\beta}_t(z)}.
\end{equation} 
If we differentiate with respect to $t$ the identity
\begin{equation}\label{equadif}
a^1_0H^{m,\beta}_z+a^2_0X^{m,\beta}_z=d_{\Phi^{m,\beta}_t(z)}\Phi^{m,\beta}_{-t}Z^{\overrightarrow{a_0}}_z(t),
\end{equation}
we get the following differential equation for $\overrightarrow{a}$:
\begin{equation}\label{linsym}
\begin{array}{rcl}
0&=&\dot{a}^1(t)H^{m,\beta}_{\Phi^{m,\beta}_t(z)}+\dot{a}^2(t)X^{m,\beta}_{\Phi^{m,\beta}_t(z)}\displaystyle\vspace{.2cm}\\
&+&a^1(t)\left[R^{m,\beta},H^{m,\beta}\right]_{\Phi^{m,\beta}_t(z)}+a^2(t)\left[R^{m,\beta},X^{m,\beta}\right]_{\Phi^{m,\beta}_t(z)}.\displaystyle
\end{array}
\end{equation}

To estimate the first Lie bracket, we observe that $m\mapsto [R^{m,\beta},H^{m,\beta}]$ is a $\Gamma(SS^2)$ valued map which is continuous in the $C^0$-topology since $m\mapsto R^{m,\beta}$ and $m\mapsto H^{m,\beta}$ are continuous in the $C^1$-topology. Hence, $\left[R^{m,\beta},H^{m,\beta}\right]=\left[R^{0,\beta},H^{0,\beta}\right]+\rho$, where $\rho$ is an $O(m)$ in the $C^0$-topology, as $m$ goes to zero. Furthermore,
\begin{eqnarray*}
\left[R^{0,\beta},H^{0,\beta}\right]=\left[V,\frac{1}{\sqrt{f}}\big(H+(\beta\circ\jmath) V\big)\right]&=&\frac{1}{\sqrt{f}}\Big([V,H]+V(\beta\circ\jmath)V\Big)\\
&=&\frac{1}{\sqrt{f}}\Big(-X-\beta V\Big)\\
&=&-X^{0,\beta}\\
&=&-X^{m,\beta}+\rho'.
\end{eqnarray*}
Putting things together, $\left[R^{m,\beta},H^{m,\beta}\right]=-X^{m,\beta}+\rho_1$, where $\rho_1$ is an $O(m)$ in the $C^0$-topology. In a similar way we find that $\left[R^{m,\beta},X^{m,\beta}\right]=H^{m,\beta}+\rho_2$. Substituting these expressions for the Lie brackets inside \eqref{linsym}, we find
\begin{equation}
\dot{a}^1H^{m,\beta}+\dot{a}^2X^{m,\beta}=a^1X^{m,\beta}-a^2H^{m,\beta}-a^1\rho_1-a^2\rho_2.
\end{equation}
Applying the trivialization $\chi^{m,\beta}$ we get
\begin{equation}\label{equadifa}
\dot{\overrightarrow{a}}=(J_{\op{st}}+\rho^{m,\beta}_z)\overrightarrow{a},
\end{equation}
where $\rho^{m,\beta}_z\in\mathfrak{gl}(2,\R)$ is of order $O(m)$.

On the other hand, differentiating Equation \eqref{lintri}, we have
\begin{equation}\label{equadifb}
\dot{\overrightarrow{a}}=\dot{\Psi}^{m,\beta}_z\overrightarrow{a}_0=\dot{\Psi}^{m,\beta}_z\left(\Psi^{m,\beta}_z\right)^{-1}\overrightarrow{a}=B^{m,\beta}_z\overrightarrow{a}.
\end{equation}
Thus, comparing \eqref{equadifa} and \eqref{equadifb}, we finally arrive at $B^{m,\beta}_z=J_{\op{st}}+\rho^{m,\beta}_z$.
\end{proof}

The previous lemma together with the following proposition reduces the condition of dynamical convexity to a condition on the period of Reeb orbits. First, for each $Z\in\Gamma(SS^2)$ we call $T_0(Z)$ the minimal period of a closed contractible orbit of $\Phi^Z$. We set $T_0(Z)=0$, if $\Phi^Z$ has a rest point. We remark that the map $Z\mapsto T_0(Z)$ is lower semicontinuous with respect to the $C^0$-topology.
\begin{prp}\label{finpro}
Let $C$ be the constant contained in \eqref{odim}. If the inequality 
\begin{equation}\label{dincopera}
 \frac{2\pi}{T_0(R^{m,\beta})}<1-Cm,
\end{equation}
is satisfied, then $\tau^{m,\beta}$ is dynamically convex.
\end{prp}
\begin{proof}
Let $\zeta$ be a contractible periodic orbit for $R^{m,\beta}$ with period $T$ and such that $\zeta(0)=z$. Consider $\Psi^{m,\beta}_z\big|_{[0,T]}\in\op{Sp}_T(1)$ defined as before and fix a $u\in\R^2\setminus\{0\}$. If $\theta_u^{\Psi^{m,\beta}_z}$ is defined by Equation \eqref{teta}, we can bound its first derivative by means of Lemma \ref{linea}, as follows:
\begin{eqnarray*}
\dot{\theta}_u^{\Psi^{m,\beta}_z}=\frac{g_{\op{st}}\big(\dot{\Psi}^{m,\beta}_zu,J_{\op{st}}\Psi^{m,\beta}_zu\big)}{|\Psi^{m,\beta}_zu|^2}&=&\frac{g_{\op{st}}\big(B^{m,\beta}_z\Psi^{m,\beta}_zu,J_{\op{st}}\Psi^{m,\beta}_zu\big)}{|\Psi^{m,\beta}_zu|^2}\\
&=&\frac{g_{\op{st}}\big((J_{\op{st}}+\rho^{m,\beta}_z)\Psi^{m,\beta}_zu,J_{\op{st}}\Psi^{m,\beta}_zu\big)}{|\Psi^{m,\beta}_zu|^2}\\
&=&\frac{|\Psi^{m,\beta}_zu|^2+g_{\op{st}}\big(\rho^{m,\beta}_z\Psi^{m,\beta}_zu,J_{\op{st}}\Psi^{m,\beta}_zu\big)}{|\Psi^{m,\beta}_zu|^2}\\
&\geq&1-\Vert \rho^{m,\beta}_z\Vert.
\end{eqnarray*}
Hence, we can estimate the normalized increment in the interval $[0,T]$ by
\begin{equation}
\Delta(\Psi^{m,\beta}_z\big|_{[0,T]},u)=\frac{1}{2\pi}\int_0^T\dot{\theta}_u^{\Psi^{m,\beta}_z}(t)dt\geq (1-Cm)\frac{T}{2\pi}.
\end{equation}
Therefore, by criterion \eqref{rmkdyn}, $\mu_{\op{CZ}}(\zeta)\geq 3$ provided $(1-Cm)\frac{T}{2\pi}>1$. Asking this condition for every contractible periodic orbit is the same as asking Inequality \eqref{dincopera} to hold. The proposition is thus proved.
\end{proof}
We are now ready to reprove Proposition \ref{corfinalo}.
\begin{proof}[Second proof of Proposition \ref{corfinalo}]
Let $p:S^3\rightarrow SS^2$ be a double cover. We have to show that $p^*\tau^{m,\beta}$ is both tight and dynamically convex.

To prove tightness, we can argue in two ways. On the one hand, we can use the fact that, if $p_{m,\beta}$ is the covering map given by Proposition \ref{convi}, we have that $p_{m,\beta}^*\tau^{m,\beta}$ is tight, since it is proportional to the standard contact form $\upsilon^*\lambda_{\op{st}}$, which is tight. Finally, there exists a diffeomorphism $F:S^3\rightarrow S^3$ such that $p_{m,\beta}=p\circ F$. A more abstract argument, independent of the discussion in the previous subsection, starts by denoting with $\xi_{\op{st}}$ the standard tight contact structure on $S^3$. Thus, if $p_0:S^3\rightarrow \R\Pro^3$ is the double cover, $(p_0)_*\xi_{\op{st}}$ is tight as well. The contact structure $\xi^{m,\beta}$ on $SS^2$ is also tight, because it is strongly fillable (see \cite{gro} and \cite{eli3}). Since there is only one tight contact structure on $SS^2$ up to isomorphism (see \cite{eli1} and \cite{etn}), $\xi^{m,\beta}$ and $(p_0)_*\xi_{\op{st}}$ are contactomorphic. This contactomorphism lifts to a contactomorphism between $p^*\xi^{m,\beta}$ and $\xi_{\op{st}}$. 

We now claim that the dynamical convexity of $p^*\tau^{m,\beta}$ and $\tau^{m,\beta}$ are equivalent. Call $\widehat{\xi}^{m,\beta}$ and $\widehat{R}^{m,\beta}$ the contact structure and the Reeb vector field of $p^*\tau^{m,\beta}$. Let $\widehat{\Phi}^{m,\beta}$ be the flow of $\widehat{R}^{m,\beta}$. Thus, $dp\circ d\widehat{\Phi}^{m,\beta}=d\Phi^{m,\beta}\circ dp$ and $\widehat{\chi}^{m,\beta}:=\chi^{m,\beta}\circ dp$ is a trivialization of $\widehat{\xi}^{m,\beta}$. Therefore,
\begin{equation}
\widehat{\chi}^{m,\beta}_{\widehat{\Phi}^{m,\beta}_t}\circ d\widehat{\Phi}^{m,\beta}_t\circ (\widehat{\chi}^{m,\beta})^{-1}=\Psi^{m,\beta}(t).
\end{equation}
Hence, $\widehat{\zeta}$ is a contractible orbit for $\widehat{R}^{m,\beta}$ if and only if $p(\widehat{\zeta})$ is a contractible orbit for $R^{m,\beta}$ and $\mu_{\op{CZ}}(\widehat{\zeta})=\mu_{\op{CZ}}(p(\widehat{\zeta}))$. This finishes the proof of the claim.

To complete the proof of the proposition, it is enough to show that inequality \eqref{dincopera} in Proposition \ref{finpro} holds for small $m$. First, we compute the periods of contractible orbits for $R^{0,\beta}=V$. A loop going around the vertical fiber $k$ times with unit angular speed has period $2\pi k$ and is contractible if and only if $k$ is even. Hence, $T_0(R^{0,\beta})=4\pi$.

Using the lower semicontinuity of the minimal period, we find that
\begin{equation}
 \limsup_{m\rightarrow 0}\left(\frac{2\pi}{T_0(R^{m,\beta})}+Cm\right)\leq \frac{2\pi}{4\pi}+0=\frac{1}{2}<1
\end{equation}
and, therefore, the inequality is still true for $m$ small emough.
\end{proof}
\begin{rmk}\label{gengen}
We now explain briefly how to modify the previous argument to get Inequality \eqref{gendyn}. For small $m$, $\omega_m$ is still of contact type and, under the normalization $\int_M\sigma=2\pi\chi(M)$, we get that $R^{m,\beta}$ is converging to $-V$, for $m$ tending to $0$. Hence, along the lines of Lemma \ref{linea}, we find that $B^{m,\beta}_z=-J_{\op{st}}+\rho^{m,\beta}_z$. Since the order of the vertical loop in $H_1(SM,\Z)$ is $|\chi(M)|$, we get that the minimal period of a null-homologous orbit of $-V$ is $2\pi|\chi(M)|$. The estimates on $B^{m,\beta}$ and $T(R^{m,\beta})$ together give an upper bound (instead of a lower bound) on the winding interval $I(\Psi^{m,\beta}_z\big|_{[0,T]})$. Finally, Inequality \eqref{gendyn} follows using the analogous of criterion \eqref{rmkdyn} when the winding interval is bounded from above.  
\end{rmk}
We end this subsection by giving a geometric proof of Inequality \eqref{dincopera} for small values of $m$ without using the lower semicontinuity of the minimal period.

We consider a finite collection of closed disks $\mathbf D:=\{D_i\ |\ D_i\subset S^2\}$ such that the open disks $\dot{\mathbf D}:=\{\dot{D}_i\}$ cover $S^2$. We also fix a collection of vector fields of unit norm $\mathbf Z=\{Z_i\ |\ Z_i\in\Gamma(D_i),\ |Z_i|=1\}$. Let $\delta$ be the Lebesgue number of the cover $\mathbf D$ with respect to the distance induced from the Riemannian metric. Finally, let $\varphi_i:SD_i\rightarrow\R/2\pi\Z$ be the angular function associated to $Z_i$. It is defined at $(x,v)\in SD_i$ by
$v=\cos\varphi_i(Z_i)_x+\sin\varphi_i(\jmath Z_i)_x$. We set

\begin{equation}
C_{(\mathbf D,\mathbf Z)}:=\sup_i\left(\sup_{D_i}\frac{|d\varphi_i(X)|}{h_{m,\beta}}\right).
\end{equation}

Let $\varepsilon>0$ be arbitrary. We claim that there exists $m_\varepsilon>0$ such that, for $m<m_\varepsilon$, the period $T$ of a contractible periodic orbit $\zeta$ of $R^{m,\beta}$ is bigger than $4\pi-\varepsilon$. This claim immediately implies \eqref{dincopera}. Suppose first that $\pi(\zeta)$ is not contained in any $D_i$. This means that $\pi(\zeta)$ is not contained in any ball of radius $\delta$. If we denote by $\ell(\pi(\zeta))$ the length of the curve, $\ell(\pi(\zeta))\geq2\delta$ and, recalling the expression of $R^{m,\beta}$,
\begin{equation}
\ell(\pi(\zeta))=\int_0^T\left|\frac{d}{dt}\pi(\zeta)\right|dt=\int_0^T\left|\frac{m}{h_{m,\beta}}\zeta\right|dt=\int_0^T\frac{m}{h_{m,\beta}}dt\leq \frac{m}{\inf h_{m,\beta}}T.
\end{equation}
Hence, if $m<\frac{2\delta\inf h_{m,\beta}}{4\pi-\varepsilon}$, then $T> 4\pi-\varepsilon$.

Suppose, on the other hand, that $\pi(\zeta)$ is contained in some $D_{i_0}$ and consider $\widetilde{\varphi}_{i_0}:[0,T]\rightarrow\R$ a lift of $\varphi_{i_0}\circ\zeta\big|_{[0,T]}:[0,T]\rightarrow\R/2\pi\Z $. Since the curve $\zeta$ is closed and contractible in $SS^2$, $\widetilde{\varphi}_{i_0}(T)-\widetilde{\varphi}_{i_0}(0)=4\pi k$, with $k\in\Z$. On the other hand,
\begin{equation}
\widetilde{\varphi}_{i_0}(T)-\widetilde{\varphi}_{i_0}(0)=\int_0^T\frac{d\varphi_{i_0}}{dt}dt=\int_0^Td\varphi_{i_0}(R^{m,\beta})dt.
\end{equation}
Moreover, $d\varphi_{i_0}(R^{m,\beta})=md\varphi_{i_0}\left(\frac{X}{h_{m,\beta}}\right)+\frac{f}{h_{m,\beta}}$. Since $\frac{f}{h_{0,\beta}}=1$, we have, for $m$ small enough,
\begin{equation}
d\varphi_{i_0}(R^{m,\beta})\geq \inf_{SS^2} \frac{f}{h_{m,\beta}}-mC_{(\mathbf D,\mathbf Z)}>0.
\end{equation}
This implies that $k>0$ and, as a consequence, that $\widetilde{\varphi}_{i_0}(T)-\widetilde{\varphi}_{i_0}(0)\geq 4\pi$. Therefore,
\begin{equation}
4\pi\leq \int_0^Td\varphi_{i_0}(R^{m,\beta})dt\leq\left(\sup_{SS^2}\frac{f}{h_{m,\beta}}+mC_{(\mathbf D,\mathbf Z)}\right)T.
\end{equation}
Using $\frac{f}{h_{0,\beta}}=1$ again, we see that there exists $m'_\varepsilon$ such that, if $m<m'_\varepsilon$,
\begin{equation}
4\pi-\varepsilon<\frac{4\pi}{\displaystyle\sup_{SS^2}\frac{f}{h_{m,\beta}}+mC_{(\mathbf D,\mathbf Z)}}\leq T.
\end{equation}
Hence, the claim follows taking $m_\varepsilon:=\min\left\{\frac{2\delta\inf h_{m,\beta}}{4\pi-\varepsilon},m'_\varepsilon\right\}$.

\endgroup
\bibliographystyle{amsalpha}
\bibliography{paper1}
\end{document}